\documentclass[final]{article}
\usepackage[margin=1in]{geometry}

\usepackage[utf8]{inputenc}
\usepackage{amsfonts}
\usepackage{amsmath,amscd,amsthm,amssymb,mathrsfs,setspace}
\usepackage{mathtools}
\usepackage{color}
\usepackage{algorithm}
\usepackage{algorithmicx}
\usepackage{algpseudocode}
\usepackage{booktabs}
\usepackage[colorlinks]{hyperref}
\usepackage[nameinlink,capitalize]{cleveref}
\usepackage{graphicx}
\usepackage{tikz}
\usepackage{subfigure}
\usepackage{url}
\usepackage{fancyhdr}
\usepackage{todonotes}
\usepackage{extarrows}
\usepackage{siunitx}
\usepackage[notref,notcite]{showkeys}
\usepackage{enumitem}

\usepackage{my_latex_commands}

\newtheorem{theorem}{Theorem}[section]
\newtheorem{proposition}[theorem]{Proposition}
\newtheorem{lemma}[theorem]{Lemma}
\newtheorem{corollary}[theorem]{Corollary}

\newtheorem{assumption}[theorem]{Assumption}

\theoremstyle{remark}
\newtheorem{remark}[theorem]{Remark}

\crefname{theorem}{Theorem}{Theorems}
\Crefname{theorem}{Theorem}{Theorems}
\crefname{assumption}{Assumption}{Assumptions}
\Crefname{assumption}{Assumption}{Assumptions}
\crefname{lemma}{Lemma}{Lemmas}
\Crefname{lemma}{Lemma}{Lemmas}
\crefname{definition}{Definition}{Definitions}
\Crefname{definition}{Definition}{Definitions}
\crefname{proposition}{Proposition}{Propositions}
\Crefname{proposition}{Proposition}{Propositions}
\crefname{algorithm}{Algorithm}{Algorithms}
\Crefname{algorithm}{Algorithm}{Algorithms}
\crefname{section}{Section}{Sections}
\Crefname{section}{Section}{Sections}
\crefname{appendix}{Appendix}{Appendices}
\Crefname{appendix}{Appendix}{Appendices}
\crefname{corollary}{Corollary}{Corollaries}
\Crefname{corollary}{Corollary}{Corollaries}
\crefname{example}{Example}{Examples}
\Crefname{example}{Example}{Examples}

\newcommand{\loc}{\textup{loc}}
\newcommand{\TV}{\mathrm{TV}}
\newcommand{\eps}{\varepsilon}
\renewcommand{\phi}{\varphi}
\newcommand*\dd{\mathop{}\!\mathrm{d}}

\DeclareMathOperator*{\BV}{BV}
\newcommand{\mres}{\mathbin{\vrule height 1.6ex depth 0pt width
		0.13ex\vrule height 0.13ex depth 0pt width 1.3ex}}
		
 \usepackage[
	backend=biber,
	style=numeric,
	maxbibnames=99,
	giveninits=true,
	sorting=nyt,
	sortcites=true,
	isbn=false,
	citecounter=true,
]{biblatex}
\title{Dual Regularization and Outer Approximation of Optimal Control Problems in BV\thanks{
		The authors acknowledge funding by Deutsche
		Forschungsgemeinschaft (DFG) under grant nos.\ BU 2313/7-1, ME 3281/12-1}}
\author{Christian Meyer\thanks{Faculty of Mathematics,
		TU Dortmund University (\url{christian2.meyer@tu-dortmund.de}, \url{annika.schiemann@tu-dortmund.de}).}
	\and 
	Annika Schiemann\footnotemark[2]}

\begin{document}
	\maketitle
	
\begin{abstract}
    This paper is concerned with an elliptic optimal control problem  with total variation (TV) restriction on the control in the constraints.
    We introduce a regularized optimal control problem by applying a quadratic regularization of the dual representation of the TV-seminorm. 
    	The regularized optimal control problem can be solved by means of an outer approximation algorithm. 
    	Convergence of the regularization for vanishing regularization parameter as well as convergence of the outer approximation algorithm is proven.
    Moreover, we derive necessary and sufficient optimality conditions for the original unregularized optimal control problem and use these to construct an exact solution
    that we use in our numerical experiments to confirm our theoretical results.
\end{abstract}

\section{Introduction} \label{sec:Intro}

This paper is concerned with the following optimal control problem:
\begin{equation}\tag{P} \label{eq:P}
    \left.
	\begin{aligned}
		\min \quad & j(y) + \frac{\alpha}{2} \| u - u_d \|_{L^2(\Omega)}^2 \\
		\text{s.t.} \quad &  y \in H^1_0(\Omega), \; u \in \UU, \\
        & - \laplace y = u + f \;\; \text{in } H^{-1}(\Omega), \quad \TV(u) \leq 1,
	\end{aligned}
	\quad \right\}
\end{equation}
where 
\begin{equation}\label{eq:defU}
    \UU \coloneqq 
    \begin{cases}
        \BV(\Omega) \cap L^2(\Omega), & \text{if } \alpha > 0,\\
        \BV(\Omega), & \text{if } \alpha = 0,
    \end{cases}        
    \quad \text{and}\quad 
    \BV(\Omega) \coloneqq \{ u \in L^1(\Omega) : \TV(u) < \infty \}.
\end{equation}
Herein, $\Omega \subset \R^d$, $d \in \N$, $d\leq 4$, is a bounded Lipschitz domain and $\TV: L^1(\Omega) \to [0,\infty]$ denotes the 
total variation defined by 
\begin{equation}
	\TV(u) = \sup \left\{ \int_{\Omega} u \, \div \varphi \dd z: \; \varphi \in C_{c}^\infty(\Omega;\R^d), \; \| \varphi \|_{L^\infty(\Omega;\R^d)} \leq 1 \right\}. \label{eq:TV_C1}
\end{equation}
As usual, $H^{-1}(\Omega)$ is short for the dual of the Sobolev space $H^1_0(\Omega)$. 
Furthermore, $j: H^1_0(\Omega) \to \R$ is assumed to be convex, lower semicontinuous, and bounded from below by a constant $\kappa > -\infty$
and $\alpha \geq 0$ denotes the Tikhonov parameter. We point out that, in some parts of the paper, $\alpha = 0$ is allowed.
Moreover, $u_d \in L^2(\Omega)$ is a given desired control and $f \in H^{-1}(\Omega)$ is a given inhomogeneity. 
Both will serve as auxiliary variables that 
allow us to construct an exact solution to assess the performance of our algorithmic approach, see Section~\ref{sec:exact_solution} below.

Let us shortly address the Poisson equation in \eqref{eq:P}. According to \cite[Theorem~10.1.3]{ABM14}, there holds $\BV(\Omega) \embed L^{d/(d-1)}(\Omega)$
with continuous embedding and thus, $\UU \embed  L^q(\Omega)$ with  $q := \min\{2,d/(d-1)\}$.
Moreover, by Sobolev embeddings, we have $H^1_0(\Omega) \embed L^{2d/(d-2)}(\Omega)$ and therefore, $L^s(\Omega) \embed H^{-1}(\Omega)$ for all 
$s \geq  2d/(d + 2)$.
Since $d \leq 4$ by assumption, we obtain $q \geq 2d/(d+2)$ and consequently $\UU \embed H^{-1}(\Omega)$. Thus, for every $u\in \UU$, 
there exists a unique solution to the state equation in \eqref{eq:P}. By denoting the solution operator of the Poisson equation 
by $S := (-\laplace)^{-1} : H^{-1} \to H^1_0(\Omega)$, we can then rewrite \eqref{eq:P} as an optimization problem in the control variable only:
\begin{equation*}
    \eqref{eq:P} 
    \quad \Longleftrightarrow \quad 
    \left\{\quad
	\begin{aligned}
		\min \quad & J(u) \coloneqq j\big(S(u+f)\big) + \frac{\alpha}{2} \| u - u_d \|_{L^2(\Omega)}^2 \\
		\text{s.t.} \quad &  u \in \UU,  \quad \TV(u) \leq 1.
	\end{aligned}
	\right.
\end{equation*}

Let us put our work into perspective.
Optimization problems with total variation regularization have a wide range of applications. This includes for example denoising in image processing \cite{ROF92, CL97, DJS07, OBDF09}.
However, total variation regularization is also of great interest in terms of optimal control, see for instance \cite{CKP98, CKK17, CKK18, Kay20}, also in combination with integrality restrictions \cite{LM22, MN23, MS23}.

To keep the discussion concise, we focus on contributions in optimal control dealing with controls in $\BV(\Omega)$. 
For the broad literature on TV-regularization in imaging, we refer to the survey articles \cite{CCCNP10, CCN15, CCY15}. 
It is to be noted that imaging problems are structurally different from optimal control problems with PDEs, 
since the forward operator in imaging usually provides a much simpler structure. Consequently, prominent methods 
like the famous Chambolle-Pock method \cite{CP16} or the fast dual proximal gradient method \cite[Section~12]{Bec17} may be inefficient
when applied to optimal control of PDEs.
For this reason, several alternative approaches for optimal control problems in $\BV(\Omega)$ have been introduced in the recent past.

There is a substantial amount of literature that deals with theoretical aspects of optimal control problems with controls in $\BV(\Omega)$, 
such as existence of optimal solutions and the derivation of necessary and sufficient optimality conditions. 
We only mention \cite{CKP98, CKP99, BBHP14, CKK18, CK19, NW22}, where optimality conditions are established for problems involving the 
TV-seminorm as Tikhonov-regularization term. 

Concerning the development of algorithms, several approaches in the literature follow the first-discretize-then-optimize paradigm, i.e., the control is discretized, 
e.g., by piecewise linear and continuous functions and the finite-dimensional problem arising in this way is solved by methods from nonsmooth optimization. 
We only refer to \cite{CKP99, HQ19, HKQ18, CKT21, CIW23} for problems with multi-dimensional control domains. 
Some of these contributions also address the convergence of the discretized problems 
and their solutions, respectively, to their infinite-dimensional counterparts for mesh size tending to zero and present corresponding error estimates. 
Concerning the piecewise constant discretization, 
this is a delicate issue in multi dimensions, since the limit does not recover the ``classical'' TV-seminorm, but an equivalent norm, where the 
Euclidian norm in the dual representation of the TV-seminorm is replaced by another term, see \cite{CKP99, CIW23} for details.
Moreover, the convergence of the optimization algorithms for the discretized problems is frequently investigated, but 
the algorithms are not considered in function space, sometimes it is even not possible to formulate the respective algorithm in function space.
For the development of function space based algorithms, the optimization problems are frequently regularized, for instance by adding 
a squared $H^1_0$-seminorm to the objective, see, e.g., \cite{CK19, EK20}, and additionally smoothing the $L^1$-norm of $\nabla u$ by 
$\sqrt{\|\nabla u\|^2 + \varepsilon}$, $\varepsilon > 0$, see, e.g., \cite{NW22, HM22}. Such a smoothing is also frequently applied after discretization, 
cf.\ e.g.\ \cite{HKQ18, BBHP14}. It behaves similar to an $L^2$-Tikhonov regularization of the pre-dual problem as shown in \cite{HS06} and
\cite{CK11}. In the latter, the $L^2$-regularized pre-dual problem is approximated in function space by means of a $H^2 \cap H^1_0$-Tikhonov term in combination 
with a penalty term for the pointwise constraints on the dual variable. 
The regularized problems arising by the various regularization techniques
are frequently solved by semismooth Newton methods, see, e.g., \cite{CK11, CKK17, HM22}. 
In particular in \cite{HM22}, such a method is investigated in detail in combination with a path-following strategy for the regularization parameter.

To summarize, all methods with guaranteed convergence approximate the control ``from the inside'', i.e., by more regular control functions that do not allow for 
discontinuities and jumps. This is for instance the case, if piecewise linear and continuous control discretizations are used or if 
the $H^1$-seminorm is added as a Tikhonov regularization term. However, the possibility of discontinuities is certainly the main
motivation for choosing $\BV(\Omega)$ as control space and therefore, it would be advantageous, if an algorithm admits discontinuous solutions as well.
In this work, we propose such an algorithm in function space, based on two basic ideas: 
\begin{enumerate}
    \item a \emph{tailored regularization of the dual problem associated with the $\TV$-seminorm} that allows to approximate the control in $\BV(\Omega)$ from ``outside'', 
    i.e., with less regular controls allowing for discontinuities,
    \item an \emph{outer approximation algorithm} for the solution of the dual regularized optimal control problems.
\end{enumerate}
It is to be noted that the idea of outer approximation is 
not new in the context of optimal control with controls in $\BV(\Omega)$ and has already been proven to be useful when 
pointwise binary resp.\ integer constraints are imposed on the control variable, see \cite{BGM24a, BGM24b, BGM24c, MS24}.

This remainder of this paper is structured as follows: 
after having fixed our notation in the rest of this introduction, 
we introduce the dual regularization of the TV-seminorm and analyze its convergence properties in Section~\ref{sec:Regularization}. 
We then introduce the regularized counterpart to \eqref{eq:P} and prove the existence of optimal solutions to \eqref{eq:P} and its regularization 
in Section~\ref{sec:exist}. 
Afterwards, Section~\ref{sec:conv} is dedicated to the convergence analysis of the regularized optimal control problems for the regularization parameter tending to zero.
In Section~\ref{sec:Outer_Approximation}, we state the outer approximation algorithm to solve the regularized control problems and investigate its convergence properties. 
In order to compare our theoretical findings with numerical results, we construct an exact solution of \eqref{eq:P} based on its first-order necessary and sufficient 
optimality conditions in Section~\ref{sec:optimality}.
Finally, in \cref{sec:numerics}, we provide two numerical examples, one with the exact solution from \cref{sec:optimality}, the other one with generic data.

\subsection*{Notation}
By $H(\div;\Omega)$, we denote the vector fields in $L^2(\Omega;\R^d)$, whose distributional divergence is an element of $L^2(\Omega)$, i.e., 
\begin{equation*}
    H(\div;\Omega) := \{ \omega \in L^2(\Omega;\R^d) \colon \div \omega \in L^2(\Omega)\}.
\end{equation*}
Equipped with the scalar product 
\begin{equation*}
    (\omega, \psi)_{H(\div;\Omega)} := \int_\Omega (\omega, \psi)_{\R^d}\,\d x + \int_\Omega \div \omega \, \div \psi\, \d x ,
\end{equation*}
this space becomes a Hilbert space.
With $H_0(\div;\Omega)$ we denote the closure of $C^\infty_c(\Omega;\R^d)$ w.r.t.\ the norm induced by the above scalar product. 
It is well known that, if $\Omega$ has a Lipschitz boundary, then vector fields in $H(\div;\Omega)$ admit a normal trace $\tau_n$ in $H^{-1/2}(\partial\Omega)$
and in this case, $H_0(\div;\Omega)$  is just the subspace of all elements in $H(\div;\Omega)$, whose normal traces vanish, cf.\ \cite[Chap.~1]{Tem79}.

If $X$ and $Y$ are linear normed spaced such that there exists a linear, bounded, and injective map from $X$ to $Y$, we say that $X$ is continuously embedded in 
$Y$ and write $X \embed Y$. If the map is in addition compact, then $X$ is compactly embedded in $Y$ and we write $X \embed^c Y$.
The space of linear and bounded operators from $X$ to $Y$ is denoted by $\LL(X, Y)$. 

Given a function $v\in L^1(\Omega)$, we set $v^0 := v - \frac{1}{|\Omega|} \int_\Omega v\,\d x$. 
The subspace of $L^1(\Omega)$ of all functions with vanishing mean is denoted by 
$L^1_0(\Omega) \coloneqq \{ v\in L^1(\Omega) \colon \frac{1}{|\Omega|} \int_\Omega v\,\d x = 0 \}$. 
Analogously we define $L^p_0(\Omega)$ for $p > 1$.
Given an integrability exponent $p$, we denote the conjugate exponent as usual by $p' := p/(p-1)$.

For a function in $\varphi \in L^\infty(\Omega;\R^d)$, we abbreviate
\begin{equation*}
    \| \varphi \|_\infty := \esssup_{x\in \Omega} \sqrt{\sum_{i=1}^d \varphi_i(x)^2 }.
\end{equation*}

The space of regular $\R^d$-valued Borel measures on the open set $\Omega$ is denoted by $\frakM(\Omega;\R^d)$. 
By the Radon-Riesz theorem it can isometrically be identified with the dual of $C_0(\Omega;\R^d)$, which is the closure of 
$C_c(\Omega;\R^d)$ w.r.t.\ the supremum norm. The total variation of an element $\mu \in \frakM(\Omega;\R^d)$ is denoted by $|\mu|(\Omega)$.
For a function $u\in \BV(\Omega)$, $\nabla u$ stands for the distributional gradient, which, in case of $\BV(\Omega)$, can be identified with an
element of $\frakM(\Omega;\R^d)$, see, e.g., \cite[Chap.~10.1]{ABM14}.

Throughout the paper, $C$ denotes a generic, non-negative constant.

\section{Dual regularization of the total variation} \label{sec:Regularization}

For the following considerations, $u$ denotes an arbitrary function in $L^1_\loc(\Omega)$, unless stated otherwise. 
In order to construct an regularization scheme for the dual representation of the TV-seminorm, we consider
a Hilbert space $\HH$ with $C^\infty_c(\Omega;\R^d) \embed \HH \embed H_0(\div;\Omega)$. 
In general, supremum in \eqref{eq:TV_C1} is of course not attained by a maximizer $\varphi \in \HH$. 
For that reason, we introduce the following regularized version of the total variation. 
To this end, we consider a symmetric, bounded, and coercive bilinear form $a: \HH \times \HH \to \R$, that is, there exist constants $\beta, \gamma >0$ such that
\begin{equation}\label{eq:acoer}
	a[\varphi,\varphi] \geq \beta \| \varphi \|^2_{\HH} \quad \forall\, \varphi \in \HH, \qquad
	a[\varphi_1,\varphi_2] \leq \gamma \, \| \varphi_1 \|_{\HH} \| \varphi_2 \|_{\HH} \quad \forall\,\varphi_1, \varphi_2 \in \HH.
\end{equation}
The dual regularization of the TV-seminorm is now obtained by adding $- \frac{\eps}{2} a[\varphi, \varphi]$ to the objective in the maximization problem \eqref{eq:TV_C1} 
with a regularization parameter $\eps >0$, which yields
\begin{align}\label{eq:TV_eps}
	\TV_\eps(u)
	= \sup \left\{F_\varepsilon(u, \varphi) : \varphi\in C^\infty_c(\Omega;\R^d), \; \|\varphi\|_\infty \leq 1 \right\}, \tag{TV$_\eps$}
\end{align}
where $F_\varepsilon: L^2(\Omega) \times \HH \to \R$ is defined by
\begin{equation}\label{eq:defFeps}
	F_\varepsilon(u, \varphi)
	:= - \frac{\varepsilon}{2}\, a[\varphi, \varphi] + \int_\Omega u \,\div \varphi\, \d x.
\end{equation}
In order to ensure that the supremum in \eqref{eq:TV_eps} is attained by a maximizer $\varphi$, at least if $u$ is sufficiently regular, we require the following

\begin{assumption}\label{assu:density}
      The set  $\{ \varphi \in C^\infty_c(\Omega;\R^d) \colon \|\varphi\|_\infty \leq 1\}$ is dense in $\{ \varphi \in \HH \colon \|\varphi\|_\infty \leq 1\}$. 
\end{assumption}

\begin{remark}\label{rem:density}
    If $\HH = H_0(\div; \Omega)$ or $\HH = H^1_0(\Omega;\R^d)$, then this density assumption is fulfilled thanks to 
    \cite[Theorem~1]{HR15}.
\end{remark} 

Assumption~\ref{assu:density} ensures that, 
in contrast to the TV-seminorm, the supremum in \eqref{eq:TV_eps} is attained by a unique maximizer $\varphi_\varepsilon \in \HH$ provided that $u\in L^2(\Omega)$. 
This easily follows from 
the direct method of calculus of variations due to the coercivity of $a$ along with the weak closedness of $\{\varphi \in \HH: \|\varphi\|_\infty \leq 1\}$.
Uniqueness follows from the strict concavity of $\varphi \mapsto F_\varepsilon(u, \varphi)$.
The existence of the maximizer immediately implies that $\TV_\eps(u) < \infty$ for all $u \in L^2(\Omega)$, whereas
the total variation $\TV(u)$ is of course not necessarily finite for all $u \in L^2(\Omega)$. For later reference, we collect these findings in the following

\begin{lemma}\label{lem:L2TVeps}
    Let Assumption~\ref{assu:density} be fulfilled. 
    Then, for every $u\in L^2(\Omega)$, it holds that $\TV_\eps(u) < \infty$ and 
    \begin{equation*}
       \TV_\eps(u) = \max \left\{F_\varepsilon(u, \varphi) : \varphi\in \HH, \; \|\varphi\|_\infty \leq 1 \right\},
    \end{equation*}        
    where the maximum is attained by a unique maximizer $\varphi_\varepsilon \in \HH$.
\end{lemma}

Let us collect some properties of the regularized TV-seminorm that will be useful in the following.

\begin{lemma}\label{lem:TVeps}
    Let $u \in L^1(\Omega)$ be given. Then the following assertions hold true:
    \begin{enumerate}[label={\normalfont(\alph*)}]
        \item\label{it:TVepsconst} $\TV_\varepsilon(u) = \TV_\varepsilon(u^0)$ and $\TV_\varepsilon(c) = 0$ for all $c\in \R$.
        \item\label{it:TVepsleqTV}
        For every $\varepsilon > 0$ there holds that  $\TV_\eps(u) \leq \TV(u)$.
        \item\label{it:TVepsmono} 
        The mapping $(0,\infty)\ni \eps \mapsto \TV_\eps(u) \in [0,\infty]$ is monotonically decreasing.
        \item\label{it:TVepsconv}
        There holds that $\TV_\eps(u) \to \TV(u)$ as $\eps \searrow 0$.
    \end{enumerate}
\end{lemma}

\begin{proof}
    Though the proofs are straight forward, we present them for convenience of the reader.

    ad \ref{it:TVepsconst}:
    For all $\varphi\in C^\infty_c(\Omega;\R^d)$ we have
    \begin{equation*}
        \int_\Omega u \, \div \varphi\, \d x - \frac{\varepsilon}{2}\, a[\varphi, \varphi] 
        =  \int_\Omega u^0 \div \varphi\, \d x - \frac{\varepsilon}{2}\, a[\varphi, \varphi], 
    \end{equation*}
    which gives the first assertion. The second follows from 
    \begin{equation*}
         \int_\Omega c \, \div \varphi\, \d x - \frac{\varepsilon}{2}\, a[\varphi, \varphi]  
         = - \frac{\varepsilon}{2}\, a[\varphi, \varphi]  \leq 0 \quad \forall\, \varphi\in C^\infty_c(\Omega;\R^d) \colon \|\varphi\|_\infty \leq 1
    \end{equation*}
    and the maximum zero is attained by $\varphi \equiv 0$.

    ad \ref{it:TVepsleqTV}: 
    Let $\varphi\in C^\infty_c(\Omega;\R^d)$ with $\|\varphi\|_\infty \leq 1$ be arbitrary. Then we have
    \begin{equation*}
    \begin{aligned}
        \int_\Omega u \, \div \varphi\, \d x - \frac{\varepsilon}{2}\, a[\varphi, \varphi] 
        \leq \int_\Omega u \, \div \varphi\, \d x \leq \TV(u)
    \end{aligned}
    \end{equation*}        
    	due to the coercivity of $a$.
    Taking the supremum over all $\varphi\in C^\infty_c(\Omega;\R^d)$ with $\|\varphi\|_\infty \leq 1$ then yields the claim.
	
	ad \ref{it:TVepsmono}:
	We argue analogously to above. Let $\varphi\in C^\infty_c(\Omega;\R^d)$ with $\|\varphi\|_\infty \leq 1$ be arbitrary and 
	let $\eps_1, \eps_2 > 0$ with $\eps_1 \leq \eps_2$ be given. 
    Then there holds	
    \begin{equation*}
    \begin{aligned}
        \int_\Omega u \div \varphi\,  \d x - \frac{\eps_2}{2} \, a[\varphi,\varphi] 
        \leq \int_\Omega u \div \varphi \, \d x - \frac{\eps_1}{2} \, a[\varphi,\varphi] \leq \TV_{\varepsilon_1}(u)
    \end{aligned}
    \end{equation*}
   Again, taking the supremum over all $\varphi\in C^\infty_c(\Omega;\R^d)$ with $\|\varphi\|_\infty \leq 1$ gives the result.
	    
    ad \ref{it:TVepsconv}: 
	Let $\{ \phi_k \}_{k \in \N} \subset C^\infty_c(\Omega;\R^d)$ be a maximizing sequence of \eqref{eq:TV_C1}, that is, there holds $\| \phi_k \|_{L^\infty(\Omega;\R^d)} \leq 1$ for all $k \in \N$ and 
	$\int_\Omega u \div \phi_k \dd x \to \TV(u)$ as $k \to \infty$. Thus $\phi_k$ is also feasible for \eqref{eq:TV_eps} and consequently, there holds
	\begin{align*}
		\TV_\eps (u) \geq F_\eps(u, \phi_k) \geq - \frac{\eps \gamma}{2} \| \phi_k \|_{\HH}^2 + \int_\Omega u \div \phi_k \dd x \quad \forall\, k\in \N
	\end{align*}
	for all $\eps >0$ by the boundedness of $a$. Consider the monotonically decreasing null sequence $\{ \eps_k \}_{k \in \N} \subset \R_{>0}$ defined by
	\begin{align*}
		\eps_k \coloneqq \frac{1}{k\, \max_{j \in \N, j \leq k} \| \phi_j \|_{\HH}^2}.
	\end{align*}
	Then, by part~\ref{it:TVepsleqTV}, there holds
	\begin{align*}
		\TV(u) \geq \limsup_{k\to\infty} \TV_{\eps_k}(u) 
		\geq \liminf_{k\to\infty} \TV_{\eps_k}(u) 
		\geq \lim_{k\to \infty} - \frac{\gamma}{2k} + \int_\Omega u \div \phi_k \dd x = \TV(u)
	\end{align*}
	as $k \to \infty$. By the monotonicity of the mapping  $\eps \mapsto \TV_\eps(u)$ by part~\ref{it:TVepsmono}, 
	there thus holds $\TV_{\eps_n}(u) \to \TV(u)$ for an arbitrary null sequence $\{\eps_n\}_{n\in\N} \subset \R_{>0}$.
\end{proof}

\begin{lemma}\label{lem:epsnormconv}
    Let Assumption~\ref{assu:density} be valid. Then, 
	for each $\varepsilon > 0$, the mapping $L^2(\Omega) \ni u \mapsto \TV_\eps(u) \in \R$ is convex and continuous.
\end{lemma}

\begin{proof}
    Though the assertion is straight forward to prove, we shortly sketch the arguments for convenience of the reader.
	Let $\eps > 0$ be arbitrary and $\lambda \in [0,1]$. Then, for each $u,\, v \in L^2(\Omega)$, there holds 
    \begin{equation}\label{eq:TVepsconvex}
    	\begin{aligned}
	    & \TV_\eps(\lambda u + (1-\lambda ) v) \\
        & \quad = \sup \left\{ \lambda F_\eps(u,\phi) + (1-\lambda)F_\eps(v,\varphi) : \varphi \in \HH, \|\varphi\|_{\infty} \leq 1 \right\} \\
	    & \quad \leq \lambda \sup \left\{ F_\eps(u,\varphi) : \varphi \in \HH, \|\varphi\|_{\infty} \leq 1 \right\} 
	    + (1-\lambda) \sup \left\{ F_\eps(v,\varphi) : \varphi \in \HH, \|\varphi\|_{\infty} \leq 1 \right\} \\
	    & \quad = \lambda \TV_\eps(u) + (1-\lambda)\TV_\eps(v),
	\end{aligned}     
    \end{equation}
	which shows the claimed convexity of the mapping.
	
	To prove its continuity, let $u_1, u_2\in L^2(\Omega)$ be arbitrary and denote the functions realizing the supremum in \eqref{eq:TV_eps} 
	according to Lemma~\ref{lem:L2TVeps} by $\varphi_i \in \HH$, $i=1,2$. 
	The necessary and sufficient optimality conditions of \eqref{eq:TV_eps} read
	\begin{align*}
		\varepsilon \, a[\varphi_i, v - \varphi_i] \geq \int_\Omega u_i \,\div (v - \varphi_i)\,\d x \quad
		\forall \,v \in \HH : \|v\|_{\infty} \leq 1
	\end{align*}
	for $i=1,2$. If one inserts $v = \varphi_2$ in the inequality for $\varphi_1$ and vice versa and adds the arising inequalities, then 
	the coercivity of $a$  along with the continuous embedding $\HH \embeds H(\div;\Omega)$ implies
	\begin{align*}
		\eps \beta \| \varphi_1 - \varphi_2 \|^2_{\HH} \leq \varepsilon\, a[\varphi_1 - \varphi_2, \varphi_1 - \varphi_2]
		\leq \int_\Omega (u_1-u_2) \div(\varphi_1-\varphi_2) \dd x  \leq C \|u_1 - u_2\|_{L^2(\Omega)} \, \|\varphi_1 - \varphi_2\|_{\HH}.
	\end{align*}
    Thus the solution mapping of the maximization problem in \eqref{eq:TV_eps} is globally Lipschitz continuous 
    (with a Lipschitz constant proportional 	to $\varepsilon^{-1}$).
	The continuity of $F_\varepsilon : L^2(\Omega) \times \HH \to \R$, $(u, \varphi) \mapsto F_\varepsilon(u, \varphi)$
	then gives the claimed continuity of $\TV_\eps$.
\end{proof}

\begin{remark}
    Similarly to \eqref{eq:TVepsconvex}, one shows that the mapping $L^2(\Omega) \ni u \mapsto \TV_\eps(u) \in \R$ satisfies the triangle inequality, i.e.,
    $\TV_\eps(u+v) \leq \TV_\eps(u) + \TV_\eps(v)$ for all $u,v\in L^2(\Omega)$. However, it is no seminorm, since it is not positive-1-homogeneous.
    Nonetheless, with a slight abuse of notation, we will frequently call $\TV_\eps(u)$ \emph{(dual) regularized TV-seminorm} in all what follows.
\end{remark}
 
If $u$ is no function from $L^2(\Omega)$, its regularized TV-seminorm is not necessarily finite so that a result as in Lemma~\ref{lem:epsnormconv} cannot be expected, 
but the following is true instead:

\begin{lemma}\label{lem:weaklyclosed}
    For every $\varepsilon > 0$ and every  $\alpha \in (-\infty, \infty]$, the sublevel set $N_\varepsilon(\alpha) := \{ u \in L^1(\Omega) \colon \TV_\varepsilon(u) \leq \alpha \}$ is convex and closed.
    The mapping $L^1(\Omega) \ni u \mapsto \TV_\varepsilon(u) \in [0,\infty]$ is convex and lower semicontinuous.
\end{lemma}

\begin{proof}
    Convexity of $N_\varepsilon(\alpha)$ and $\TV_\varepsilon : L^1(\Omega) \to [0, \infty]$  
    follows by the same estimate as in \eqref{eq:TVepsconvex}. To show the closedness of $N_\varepsilon(\alpha)$, consider a sequence $\{u_n\} \subset N_\varepsilon(\alpha)$ 
    converging to $u$ in $L^1(\Omega)$. Then we have for every $\varphi \in C_c^\infty(\Omega;\R^d)$ with $\|\varphi\|_\infty \leq 1$ that 
    \begin{equation*}
        \alpha \geq \int_\Omega u_n \,\div \varphi \,\d x - \frac{\varepsilon}{2}\, a[\varphi, \varphi] 
        \to \int_\Omega u \,\div \varphi \,\d x - \frac{\varepsilon}{2}\, a[\varphi, \varphi] \quad \text{as } n \to \infty
    \end{equation*}
    and thus $u \in N_\varepsilon(\alpha)$. Since its sublevel sets are closed for every $\alpha$, $\TV_\varepsilon$ is lower semicontinuous.
\end{proof} 

It is well known that the distributional gradient of a function in $\BV(\Omega)$ is a regular Borel measure. For reasons of comparison, 
let us shortly address the connection between the regularized TV-seminorm and the distributional gradient. 
 
\begin{lemma}\label{lem:TVepsDu}
    Suppose that $C^\infty_c(\Omega;\R^d)$ is dense in $C_0(\Omega;\R^d) \cap \HH$. 
    Then a function $u\in L^1_{\textup{loc}}(\Omega)$ satisfies $\TV_\varepsilon(u) < \infty$, if and only if there exists $\mu \in \frakM(\Omega;\R^d)$ and $\ell \in \HH^*$ such that 
    the distributional gradient of $u$, denoted by $D u$, satisfies
    \begin{equation}\label{eq:distrderiv}
        - D u(\varphi) = \int_\Omega \varphi \cdot \d \mu + \dual{\ell}{\varphi} \quad \forall\, \varphi \in C^\infty_c(\Omega;\R^d),
    \end{equation}
    i.e., $D u \in \frakM(\Omega;\R^d) + \HH^*$.
\end{lemma} 
 
\begin{proof}
    If $u \in L^1_{\textup{loc}}(\Omega)$ satisfies \eqref{eq:distrderiv}, then, for all $\varphi\in C^\infty_c(\Omega;\R^d)$ with $\|\varphi \|_\infty \leq 1$, 
    the coercivity of the bilinear form $a$ implies
    \begin{equation*}
    \begin{aligned}
        \int_\Omega u \,\div \varphi \,\d x - \frac{\varepsilon}{2}\, a[\varphi, \varphi] &=  \int_\Omega \varphi \cdot \d \mu + \dual{\ell}{\varphi} - \frac{\varepsilon}{2}\, a[\varphi, \varphi] \\
        & \leq |\mu|(\Omega) + \frac{1}{2\varepsilon\beta}\, \|\ell\|_{\HH^*}^2 + \frac{\varepsilon}{2}\,\beta\,\|\varphi\|_\HH^2 - \frac{\varepsilon}{2}\, a[\varphi, \varphi]
        \leq |\mu|(\Omega) + \frac{1}{2\varepsilon\beta}\, \|\ell\|_{\HH^*}^2,
    \end{aligned}
    \end{equation*}
    which implies $\TV_\varepsilon(u) \leq  |\mu|(\Omega) + \frac{1}{2\varepsilon\beta}\, \|\ell\|_{\HH^*}^2 < \infty$.

    To prove the reverse implication, assume that $\TV_\varepsilon(u) < \infty$ such that the distributional derivative of $u$ satisfies
    \begin{equation}\label{eq:Duest}
    \begin{aligned}
        - \dual{Du}{\varphi} = \int_\Omega u \, \div \varphi\, \d x \leq \TV_\varepsilon(u) + \frac{\varepsilon}{2}\, a[\varphi, \varphi]
        \leq \TV_\varepsilon(u) + \frac{\varepsilon}{2}\,\gamma\, \|\varphi\|_\HH^2 \leq C \qquad\qquad & \\
        \forall\, \varphi \in  C^\infty_c(\Omega;\R^d) \colon \|\varphi \|_\infty  + \|\varphi\|_\HH \leq 1. &
    \end{aligned}
    \end{equation}
    Therefore, due to the density of $C^\infty_c(\Omega;\R^d)$ in $C_0(\Omega;\R^d) \cap \HH$, we can extend $Du$ to an element of the dual space
    $(C_0(\Omega;\R^d) \cap \HH)^* = \frakM(\Omega;\R^d) + \HH^*$ as claimed. 
\end{proof}

\begin{proposition}\label{prop:Lqest}
    Let $H^1_0(\Omega;\R^d) \embed \HH$. Then, for every $1 < q < d/(d-1)$ and every $u\in L^1(\Omega)$, there holds that 
    \begin{equation}\label{eq:Lqest}
        \| u \|_{L^q(\Omega) / \R} \leq C \big( \TV_\varepsilon (u) + \varepsilon \big)
    \end{equation}
    with a constant $C>0$ that is independent of $\varepsilon$ and $u$.
\end{proposition}

\begin{proof}
    Let $v \in C_c^\infty(\Omega)$ with $\|v\|_{L^{q'}(\Omega)}\leq 1$ be arbitrary. To construct a suitable test function for the regularized TV-seminorm, 
    we define $\varphi := B v^0$, where $B : L^{q'}_0(\Omega) \to W^{1,q'}_0(\Omega;\R^d)$ is the Bogovskii operator from \cite{Bog79}, 
    which satisfies $\div B v = v$ for all $v\in L^{q'}_0(\Omega)$ and is linear and continuous provided that $\Omega$ is a bounded Lipschitz domain, 
    see e.g.\ \cite{CMcI10}. Let us denote the associated operator norm by $\|B\|$.
    Then, due to $q' > d$, Sobolev embedding theorems imply that
    \begin{equation}\label{eq:bogovskii}
    \begin{aligned}
        \|\varphi\|_\infty + \|\varphi\|_{H^1(\Omega;\R^d)} 
        & \leq C_1 \,\| \varphi\|_{W^{1,q'}(\Omega;\R^d)} \\
        & \leq C_1 \, \|B\| \, \|v^0\|_{L^{q'}(\Omega)} \leq C_1\, \|B\|\, \big(1 + |\Omega|^{1/q'} \big) \|v\|_{L^{q'}(\Omega)} 
        =: C_2 \, \|v\|_{L^{q'}(\Omega)} 
    \end{aligned}
    \end{equation}        
    Now since $v \in C^\infty_c(\Omega)$, the mapping properties of the Bogovskii operator imply that $\varphi\in C^\infty_c(\Omega;\R^d)$, too, cf.\ e.g.\ \cite[Theorem~6.6]{DRS10}. 
    Therefore, thanks to \eqref{eq:bogovskii}, the function $\tilde\varphi$ defined by
    \begin{equation*}
        \tilde\varphi := \frac{1}{C_2 \, \|v\|_{L^{q'}(\Omega)}} \, \varphi
    \end{equation*}
    is feasible for the maximization problem in \eqref{eq:TV_eps}, which results in 
    \begin{equation}\label{eq:u0est}
    \begin{aligned}
        \int_\Omega u^0\, v\,\d x = \int_\Omega u^0\, v^0\,\d x
        & =  C_2 \,\|v\|_{L^{q'}(\Omega)}
        \Big( \int_\Omega u^0\, \div \tilde \varphi \,\d x - \frac{\varepsilon}{2}\, a[\tilde\varphi, \tilde\varphi] + \frac{\varepsilon}{2}\, a[\tilde\varphi,\tilde\varphi] \Big) \\
        & \leq C_2 \Big( \TV_\varepsilon(u^0) + \frac{\gamma\, C_\HH^2}{2\,C_2^2\,\|v\|_{L^{q'}(\Omega)}^2 }\, \varepsilon\, \|\varphi\|_{H^1(\Omega;\R^d)}^2 \Big) \\
        & \leq C_2 \Big(  \TV_\varepsilon(u^0) + \frac{\gamma\, C_\HH^2}{2}\, \varepsilon \Big),
    \end{aligned}
    \end{equation}
    where $C_\HH$ is the embedding constant of $H^1_0(\Omega;\R^d) \embed \HH$.
    The desired estimate then follows from 
    \begin{equation*}
        \|u^0\|_{L^{q}(\Omega)} 
       = \sup\Big\{  \int_\Omega u^0\, v\,\d x \colon v \in C^\infty_c(\Omega),\; \|v\|_{L^{q'}(\Omega)} \leq 1\Big\}
       \leq C_2 \Big(  \TV_\varepsilon(u^0) + \frac{\gamma\, C_\HH^2}{2}\, \varepsilon \Big), 
    \end{equation*}
    which finishes the proof.    
\end{proof}

\begin{remark}\label{rem:limitcase}
    Proposition~\ref{prop:Lqest} shows that a bounded regularized TV-seminorm implies a bound in the quotient space $L^{q}(\Omega) / \R$ 
    for every $q < d/(d-1)$. On the other hand, the TV-seminorm (without regularization)
    induces a bound in $L^{d/(d-1)}(\Omega) / \R$, see e.g.\ \cite[Theorem~10.1.3]{ABM14}. It is an open question, if 
    the limit case $q = d/(d-1)$ also holds true with regularization, i.e., if a bound of the regularized TV-seminorm implies a bound of the $L^{d/(d-1)}(\Omega) / \R$-norm.
\end{remark}

\section{Existence of optimal controls}\label{sec:exist}

Given the regularized total variation, we consider the following regularized version of \eqref{eq:P}:
\begin{equation} \tag{P$_\varepsilon$} \label{eq:Peps}
    \left.	
	\begin{aligned}
		\mbox{min } \quad & J(u) \\ 
		\mbox{s.t. } \quad & \TV_\eps(u) \leq 1
	\end{aligned}
	\quad \right\}
\end{equation}

\begin{remark}
    As seen in Lemma~\ref{lem:L2TVeps}, the regularized TV-seminorm is finite for every control in $L^2(\Omega)$. Moreover, as an immediate 
    consequence of Lemma~\ref{lem:TVeps}\ref{it:TVepsleqTV}, every control $u \in \BV(\Omega) \cap L^2(\Omega)$ that is feasible for \eqref{eq:P} 
    is automatically feasible for \eqref{eq:Peps}, too. We have therefore enlarged the set of feasible controls by introducing the dual regularized TV-seminorm.
    This illustrates that we indeed approximate the original problem \eqref{eq:P} ``from the outside'' as claimed in the introduction.
\end{remark}

We first address the existence of optimal solutions to \eqref{eq:Peps}. In case of a positive Tikhonov parameter $\alpha$, the proof is straight forward. 
If $\alpha = 0$, then we can employ Proposition~\ref{prop:Lqest} to verify the existence of optimal controls under an additional coercivity assumption on 
the objective, see Theorem~\ref{thm:existencewoTikhonov} below.

\begin{proposition}\label{prop:mitTikhonov}
    If the Tikhonov parameter satisfies $\alpha > 0$, then, for every $\varepsilon > 0$, \eqref{eq:Peps} admits a unique optimal control in $L^2(\Omega)$.
\end{proposition}

\begin{proof}
    The proof follows standard arguments. The objective is weakly lower semicontinuous and, due to $\alpha > 0$, 
    strictly convex and coercive. By Lemma~\ref{lem:weaklyclosed}, the feasible set is weakly closed in $L^2(\Omega)$. The assertion thus directly follows 
    with the direct method of calculus of variations.
\end{proof}

\begin{theorem}\label{thm:existencewoTikhonov}
    Let $d \leq 3$ and assume that the objective $j$ is radially unbounded in the following sense: For all $\hat y \in H^1_0(\Omega) \setminus \{0\}$ and all $R > 0$, there holds that
    \begin{equation}\label{eq:jcoercive}
        \inf \big\{ j ( c\, \hat y + w) \colon w \in H^1_0(\Omega),\; \|w\|_{H^1_0(\Omega)} \leq R \big\} \to \infty \quad \text{as } |c| \to \infty.
    \end{equation}
    Suppose further that  $H^1_0(\Omega;\R^d) \embed \HH$ and choose an integrability exponent $q \in [2d/(d+2), d/(d-1)[$ . 
    Then, for every $\varepsilon > 0$, the optimal control problem 
    \begin{equation}\tag{P$_\varepsilon^0$}\label{eq:ohneTikhonov}
        \left.
        \begin{aligned}
            \min \quad & j(y) \\
            \text{\textup{s.t.}} \quad & u \in L^{q}(\Omega), \quad y \in H^1_0(\Omega), \\
            & - \laplace y = u + f \quad \text{\textup{in }} H^{-1}(\Omega), \quad \TV_\varepsilon(u) \leq 1
        \end{aligned}
        \quad \right\}
    \end{equation}
    admits an optimal solution. The solution is unique provided that $j$ is strictly convex.
\end{theorem}

\begin{proof}
    The assertion is more or less a direct consequence of Proposition~\ref{prop:Lqest}. 
    In order to use this result, we rewrite \eqref{eq:ohneTikhonov} as
    \begin{equation*}
        \eqref{eq:ohneTikhonov}
        \quad \Longleftrightarrow\quad 
        \left\{ \quad
        \begin{aligned}
            \min \quad & j\big( S (u^0 + c + f) \big) \\
            \text{s.t.} \quad & u^0 \in L^{q}(\Omega), \; c\in \R,\; \int_\Omega u^0 \,\d x = 0, \; \TV_\varepsilon(u^0 + c) \leq 1,
        \end{aligned}
        \right.
    \end{equation*}
    where again $S = (-\laplace)^{-1} : H^{-1}(\Omega) \to H^1_0(\Omega)$ denotes the solution operator of Poisson's equation. 
    Note that, due to $q \geq 2d/(d+2)$, we have $L^{q}(\Omega) \embed H^{-1}(\Omega)$, since $d\leq 4$, so that $S u^0$ is well defined.
    To show the boundedness of the constant component $c$, we apply \eqref{eq:jcoercive} with $\hat y = S \eins$, where 
    $\eins$ is the function defined by $\eins(x) \equiv 1$ for all $x\in \Omega$. To this end, observe that, for all feasible $u = u^0 + c$, 
    Proposition~\ref{prop:Lqest} together with Lemma~\ref{lem:TVeps}\ref{it:TVepsconst} implies that  
    \begin{equation}\label{eq:statebound}
    \begin{aligned}
        \| S (u^0 + f)\|_{H^1_0(\Omega)} 
        & \leq  \|S\|_{\LL(L^{q}(\Omega), H^1_0(\Omega))}\, \|u^0\|_{L^q(\Omega)}  +  \| S f\|_{H^1_0(\Omega)} \\
        & \leq C (\TV_\varepsilon(u^0) + \varepsilon)  +  \| S f\|_{H^1_0(\Omega)} \\
        & = C (\TV_\varepsilon(u) + \varepsilon)  +  \| S f\|_{H^1_0(\Omega)} 
        \leq C (1 + \varepsilon)  +  \| S f\|_{H^1_0(\Omega)} =: R.    
    \end{aligned}
    \end{equation}        
    Therefore, if we consider an arbitrary infimal (and thus in particular feasible) sequence $\{u^0_n + c_n\}$, then the sequence $\{u_n^0\}$ is bounded in $L^{q}(\Omega)$ 
    by Proposition~\ref{prop:Lqest}, while the radial unboundedness assumption in \eqref{eq:jcoercive} with $\hat y = S \eins$ and $R$ as defined above 
    implies that the sequence $\{c_n\}$ is bounded. Thus every infimal sequence is bounded in $L^{q}(\Omega)$. 
    By Lemma~\ref{lem:weaklyclosed}, the feasible set $\{ u \in L^q(\Omega) \colon \TV_\varepsilon(u) \leq 1\}$ is convex and closed, 
    thus weakly closed in $L^{q}(\Omega)$. Together with the weak lower semicontinuity of $j$ by assumption, this implies that 
    the standard direct method of calculus of variations gives the existence of at least one globally optimal control.
    If $j$ is in addition strictly convex, then the injectivity of $S$ implies that the minimizer is unique.
\end{proof}

\begin{remark}\label{rem:tracking}
    The assumption in \eqref{eq:jcoercive} is fulfilled by prominent examples such as tracking type objectives of the form 
    \begin{equation*}
        j_2(y) = \frac{1}{2}\, \| y - y_d\|_{L^2(\Omega)}^2, \quad j_{1,2}(y) := \frac{1}{2}\, \| \nabla y - z\|_{L^2(\Omega;\R^d)}^2,
    \end{equation*}
    as one easily verifies.
\end{remark}

In case of the original unregularized problem \eqref{eq:P} completely the same arguments apply. For a unified 
presentation, we collect the assumptions of Proposition~\ref{prop:mitTikhonov} and Theorem~\ref{thm:existencewoTikhonov} 
in the following

\begin{assumption}\label{assu:existence}
    In order to ensure the existence of a unique solution to \eqref{eq:Peps}, we assume that 
    \begin{enumerate}[label={\normalfont(\alph*)}]
        \item\label{it:mitTik} either $\alpha > 0$,
        \item\label{it:ohneTik} or $d\leq 3$, $H^1_0(\Omega;\R^d)\embed \HH$, and $j$ satisfies the coercivity condition in \eqref{eq:jcoercive} and is strictly convex.
    \end{enumerate}        
\end{assumption}

\begin{corollary}\label{cor:exP}
    Assume that Assumption~\ref{assu:existence} holds. Then there exists a unique solution $u^*$ to 
    the unregularized optimal control problem \eqref{eq:P}.
\end{corollary}

\begin{proof}
    Since the set $\{u \in \BV(\Omega) : \TV(u) \leq 1\}$ gives a bound in $L^q(\Omega)$ and is 
    convex and closed w.r.t.\ strong convergence in $L^1(\Omega)$, we can argue in exactly the same way as in the proofs of 
    Proposition~\ref{prop:mitTikhonov} and Theorem~\ref{thm:existencewoTikhonov}.
\end{proof}

\begin{remark}
    It is to be noted that the assumption $d \leq 3$ in Theorem~\ref{thm:existencewoTikhonov} is only needed to ensure that the interval 
    for the integrability exponent $q$, i.e., $[2d/(d+2), d/(d-1)[$\,, is non-empty. If Proposition~\ref{prop:Lqest} also holds for the limit case $d/(d-1)$, 
    then the assertion of Theorem~\ref{thm:existencewoTikhonov} would be true for $d=4$, too. 
    Note moreover that, if $j$ is well defined for states in $L^2(\Omega)$, then there is no need to consider $S = (-\laplace)^{-1}$ as an operator in 
    $H^1_0(\Omega)$ and we can work with less regular Sobolev spaces in order to decrease the lower bound for $q$. 
    Thus the existence result from Theorem~\ref{thm:existencewoTikhonov} is also true for spatial dimensions larger than 3, if $j: L^2(\Omega) \to \R$.
    Since $d=2,3$ are the relevant cases for practical applications, we do not elaborate this in detail for the purpose of a concise presentation.
    
    Note moreover that, in case without regularization, the assertion of Proposition~\ref{prop:Lqest} also holds in the limit case, 
    i.e., $\BV(\Omega) \embed L^{d/(d-1)}(\Omega)$, see
    Remark~\ref{rem:limitcase}. Thus the assertion of Corollary~\ref{cor:exP} also holds, if $d=4$, $\alpha = 0$, 
    $H^1_0(\Omega;\R^d)\embed \HH$, and $j$ satisfies \eqref{eq:jcoercive} and is strictly convex.
\end{remark}

\section{Convergence of the dual regularization}\label{sec:conv}

In order to show that the dual regularization approach makes sense in the context of optimal control,
we study the convergence behavior of the optimal solutions to \eqref{eq:Peps} as the regularization parameter $\eps$ is driven to zero.
For this purpose, let $\varepsilon \searrow 0$ and denote the unique optimal solution to \eqref{eq:Peps} by $ u_\varepsilon$. 
If the existence of solutions is guaranteed by Assumption~\ref{assu:existence}\ref{it:ohneTik}, then we always take the same integrability exponent $q$ 
from Theorem~\ref{thm:existencewoTikhonov} for every value of $\varepsilon$. 
In order to unify the presentation, we moreover set 
\begin{equation}\label{eq:def_s}
    s := 
    \begin{cases}
        2, & \text{ in case of Assumption~\ref{assu:existence}\ref{it:mitTik}}, \\
        q < d/(d-1), & \text{ in case of Assumption~\ref{assu:existence}\ref{it:ohneTik}}.
    \end{cases}
\end{equation}

\begin{lemma}\label{lem: TV_weaklimit}
    Suppose that Assumption~\ref{assu:existence} is fulfilled such that \eqref{eq:Peps} admits a unique solution. 
	Let $\{ u_{\eps} \}_{\eps >0}$ denote a sequence of such solutions with $u_\eps \rightharpoonup \bar{u}$ in $L^s(\Omega)$ as $\eps \searrow 0$. 
	Then the weak limit $\bar{u}$ satisfies $\TV(\bar{u}) \leq 1$.
\end{lemma}

\begin{proof}
	Let $\eps_1 >0$ be fixed but arbitrary.
	Then, for each $\eps_2>0$ with $\eps_2 < \eps_1$, there holds $\TV_{\eps_1}(u_{\eps_2}) \leq \TV_{\eps_2}(u_{\eps_2}) \leq 1$ by the monotonicity 
	according to Lemma~\ref{lem:TVeps}\ref{it:TVepsmono} and the feasibility of $u_{\eps_2}$ for (P$_{\eps_2}$).
	By Lemma~\ref{lem:weaklyclosed}, $\TV_{\eps_1}(\cdot)$ is convex and lower semicontinuous, and thus weakly lower semicontinuous. Therefore, we obtain
	\begin{align*}
		\TV_{\eps_1}(\bar{u})
		\leq \liminf_{\eps \searrow 0} \TV_{\eps_1}(u_\eps)
		\leq \limsup_{\eps \searrow 0} \TV_{\eps_1}(u_\eps) \leq 1.
	\end{align*}
	Since $\eps_1 > 0$ was arbitrary, there holds $\TV_\eps(\bar{u}) \leq 1$ for each $\eps >0$.
	Thus, we have for every $\varphi \in C^\infty_c(\Omega;\R^d)$ with $\|\varphi\|_\infty \leq 1$ that 
	\begin{align*}
		\int_\Omega \bar{u} \,\div \varphi\, \d x        
		\leq \TV_\eps(\bar{u})+ \frac{\eps}{2}\, a[\varphi, \varphi]
		\leq 1 + \frac{\eps}{2}\, a[\varphi, \varphi] \to 1
	\end{align*}
	as $\eps \searrow 0$ and by supremizing over all $\phi \in C^\infty_c(\Omega;\R^d)$ with $\| \phi \|_\infty \leq 1$, we obtain that $\TV(\bar{u}) \leq 1$ as claimed.
\end{proof}

Now we can argue exactly as in the proof of \Cref{thm:conveps} to show the following:

\begin{theorem}\label{thm:convTV}
    Let Assumption~\ref{assu:existence} be satisfied and denote by $\{u_\varepsilon\}_{\varepsilon>0}$ the sequence of the unique optimal solutions to \eqref{eq:Peps} 
    for $\varepsilon \searrow 0$. Then the whole sequence $\{u_\varepsilon\}_{\eps > 0}$ converges weakly in $L^s(\Omega)$ to the unique minimizer of \eqref{eq:P}. 
    If $\alpha > 0$, then the sequence converges strongly in $L^2(\Omega)$.
\end{theorem}

\begin{proof}		
    First of all, the sequence $\{u_\varepsilon\}_{\varepsilon>0}$ is bounded in $L^s(\Omega)$. In case of $\alpha > 0$, this follows 
    directly from the boundedness of the objective, which is a consequence of the feasibility of $u \equiv 0$ for \eqref{eq:Peps} for all $\varepsilon > 0$:
	\begin{equation}\label{eq:ubound}
	\begin{aligned}
        \| u_\eps \|_{L^2(\Omega)} 
		& \leq \| u_\eps - u_d \|_{L^2(\Omega)} + \| u_d \|_{L^2(\Omega)} \\
        & \leq\sqrt{ \frac{2}{\alpha} (J(u_\eps) - \kappa)} + \| u_d \|_{L^2(\Omega)} \leq \sqrt{\frac{2}{\alpha} (J(0) - \kappa)} + \| u_d \|_{L^2(\Omega)} < \infty,
	\end{aligned}
	\end{equation}
	where $\kappa$ is the lower bound of the state-dependent part $j$ of the objective. 
    In case $\alpha = 0$, where the existence is guaranteed by Assumption~\ref{assu:existence}\ref{it:ohneTik}, the boundedness of $\{u_\varepsilon\}_{\varepsilon>0}$
    in $L^q(\Omega)$ follows from the constraint $\TV_\varepsilon(u)\leq 1$ as in the proof of Theorem~\ref{thm:existencewoTikhonov}: 
    While the part $u_\varepsilon^0$ with vanishing mean value is bounded in $L^q(\Omega)$ by Proposition~\ref{prop:Lqest}, the mean value $\bar u_\varepsilon$
    is again bounded by the coercivity assumption on $j$ in \eqref{eq:jcoercive}, which is seen as follows:  
    Since $u \equiv 0$ is feasible for \eqref{eq:Peps} for all $\varepsilon > 0$, there holds
    \begin{equation}\label{eq:jbound}
        j(S(u^0_\varepsilon + \bar u_\varepsilon  + f)) \leq j(S f) < \infty \quad  \forall\, \varepsilon > 0.
    \end{equation}       	
	Moreover, analogously to \eqref{eq:statebound}, we find
	\begin{equation*}
	    \| S (u_\varepsilon^0 + f)\|_{H^1_0(\Omega)} 
        \leq C (\TV\varepsilon(u) + \varepsilon)  +  \| S f\|_{H^1_0(\Omega)} 
        \leq 2\,C +  \| S f\|_{H^1_0(\Omega)} =: R \quad \forall\,\varepsilon \in (0,1].
	\end{equation*}
    Thus, the coercivity condition in \eqref{eq:jcoercive} with $\hat y = S\eins$ implies that $\bar u_\varepsilon$ must be bounded, too, since 
    otherwise $j(S(u^0_\varepsilon + \bar u_\varepsilon  + f))$ would tend to infinity contradicting \eqref{eq:jbound}.
	
	Therefore, the sequence $\{u_\eps\}_{\eps >0}$ admits a weakly converging subsequence in $L^s(\Omega)$, which we denote by the 
	same symbol for simplicity, that is, $u_\varepsilon \rightharpoonup u^*$ in $L^s(\Omega)$ for some $u^* \in L^s(\Omega)$. 
	By \cref{lem: TV_weaklimit}, there holds $\TV(u^*) \leq 1$ such that $u^*$ is feasible for the limit problem \eqref{eq:P}.
	To show its optimality, let $u \in \BV(\Omega)$ with $\TV(u)\leq 1$ be arbitrary. Then, 
	by Lemma~\ref{lem:TVeps}\ref{it:TVepsleqTV}, we obtain that $\TV_\eps(u) \leq 1$, i.e., $u$ is also feasible for \eqref{eq:Peps} for each $\eps >0$. 
    Together with the weak lower semicontinuity of $J$, this implies
    \begin{equation}\label{eq:convobj}
		J(u^*) \leq \liminf_{\eps \searrow 0} J(u_\eps) \leq \limsup_{\eps \searrow 0} J(u_\eps) \leq J(u),
    \end{equation}    	
	which yields the optimality of $u^*$. Since \eqref{eq:P} is uniquely solvable, the Urysohn subsequence principle implies the weak convergence of 
	the whole subsequence as claimed.
	
    To show the strong convergence in $L^2(\Omega)$ in case of $\alpha > 0$, first observe that, again by Lemma~\ref{lem:TVeps}\ref{it:TVepsleqTV},
    $u^*$ is feasible for \eqref{eq:Peps} for every $\varepsilon > 0$. 	Therefore, \eqref{eq:convobj} also holds with $u = u^*$, which implies the 
    convergence of the objective, i.e., 
    \begin{equation*}
         j\big(S (u_\varepsilon + f) \big) + \frac{\alpha}{2}\, \| u_\varepsilon - u_d \|_{L^2(\Omega)}^2 
         \to  j\big(S (u^* + f) \big) + \frac{\alpha}{2}\, \| u^* - u_d \|_{L^2(\Omega)}^2 
    \end{equation*}
    Since both components, i.e., $u \mapsto j(S (u_\varepsilon + f))$	and $u\mapsto \frac{\alpha}{2}\, \| u - u_d \|_{L^2(\Omega)}^2 $, 
    are weakly lower semicontinuous by our assumptions on $j$ and the linearity of $S$, this implies that 
    they converge separately, too, i.e., in particular
    \begin{equation*}
        \frac{\alpha}{2}\,  \| u_\varepsilon - u_d \|_{L^2(\Omega)}^2 \to \frac{\alpha}{2}\, \| u^* - u_d \|_{L^2(\Omega)}^2, 
    \end{equation*}
    which in turn implies that $\|u_\varepsilon\|_{L^2(\Omega)}$ converges to  $\|u^*\|_{L^2(\Omega)}$. 
    Since weak and norm convergence imply strong convergence, this finishes the proof.
\end{proof}

\subsection{Convergence rate under additional assumptions}

In the following, we will derive a convergence w.r.t.\ the regularization parameter $\varepsilon$ under the following additional assumption on the objective in \eqref{eq:P}:

\begin{assumption}\label{assu:Lsmooth}
    We assume the following conditions in addition to our standing assumptions:
    \begin{enumerate}[label={\normalfont(\alph*)}]
        \item $\alpha > 0$.
        \item\label{it:Lsmooth} The state-dependent part $j$ of the objective is G\^ateaux-differentiable from $H^1_0(\Omega)$ to $\R$ 
        and its G\^ateaux-derivative is Lipschitz continuous with Lipschitz constant $L_j$, i.e., 
        \begin{equation*}
            \| j'(y_1) - j'(y_2)\|_{H^{-1}(\Omega)} \leq L_j \, \| y_1 - y_2\|_{H^1_0(\Omega)} 
            \quad \forall\, y_1, y_2 \in H^1_0(\Omega).
        \end{equation*}
     \end{enumerate}        
\end{assumption}

\begin{remark}
    It is readily verified that the two prominent examples from Remark~\ref{rem:tracking} satisfy Assumption~\ref{assu:Lsmooth}\ref{it:Lsmooth}.
\end{remark}

Under Assumption~\ref{assu:Lsmooth} the objective $J(u) = j(S(u+f)) + \frac{\alpha}{2}\, \|u - u_d\|_{L^2(\Omega)}^2$ is G\^ateaux differentiable 
from $L^2(\Omega)$ to $\R$ and its derivative is Lipschitz continuous with Lipschitz constant 
\begin{equation*}
    L_J := \| S \|_{\LL(L^2(\Omega), H^1_0(\Omega))}^2 \, L_j + \alpha.
\end{equation*}
Therefore, by \cite[Theorem~7.3]{CV20}, its Fenchel conjugate $J^* : L^2(\Omega) \to \overline{\R}$ is strongly convex with constant $L_J^{-1}$.
Moreover, as a conjugate functional, $J^*$ is lower semicontinuous and, due to $\alpha > 0$, there holds that $\dom(J^*) = L^2(\Omega)$, 
i.e., $J^*$ is finite.

\begin{lemma}\label{lem:barPhi}
    Under Assumption~\ref{assu:Lsmooth} the minimization problem 
    \begin{equation}\tag{PD$_\infty$}\label{eq:priminf}
        \left.
        \begin{aligned}
            \min \quad & J^*(\div\Phi) + \|\Phi\|_\infty \\
            \text{s.t.} \quad &  \Phi \in H_0(\div;\Omega) \cap L^\infty(\Omega;\R^d)   
        \end{aligned}
        \quad \right\}
    \end{equation}
    admits a solution $\bar\Phi \in H_0(\div;\Omega) \cap L^\infty(\Omega;\R^d)$.
\end{lemma}

\begin{proof}
    Let us denote the objective in \eqref{eq:priminf} by $q$ and fix an arbitrary point $\Phi_0 \in H_0(\div;\Omega) \cap L^\infty(\Omega;\R^d)$. Since $J^*$ is strongly convex, 
    there holds that 
    \begin{equation}\label{eq:dualobjradunbound}
        q(\Phi) \geq q(\Phi_0) + \int_\Omega g \, \div (\Phi - \Phi_0) + \frac{1}{2L_J} \|\div \Phi - \div \Phi_0\|_{L^2}^2 + \|\Phi\|_\infty
        \quad\forall\, \Phi \in H_0(\div;\Omega) \cap L^\infty(\Omega;\R^d),
    \end{equation}
    where $g\in L^2(\Omega)$ is an arbitrary element of $\partial J^*(\div \Phi_0)$. Note that the chain rule for convex subdifferentials is applicable, since 
    $\dom(J^*) = L^2(\Omega)$. Since the right hand side of \eqref{eq:dualobjradunbound} is radially unbounded on $H_0(\div;\Omega) \cap L^\infty(\Omega;\R^d)$, 
    so is $q$. Thus, w.l.o.g., we can assume that every minimizing sequence admits a subsequence converging weakly in $H_0(\div;\Omega)$ and 
    weakly-* in $L^\infty(\Omega;\R^d)$. Since the objective is lower semicontinuous w.r.t.\ this convergence (as $J^*$ is convex and lower semicontinuous),
    the classical direct method of calculus of variations gives the existence of a minimizer. 
\end{proof}

We now consider problem \eqref{eq:priminf} in the space $X := H_0(\div;\Omega) \cap C_0(\Omega;\R^d)$, i.e., 
\begin{equation}\tag{PD$_C$}\label{eq:primC}
    \left.
    \begin{aligned}
        \inf \quad & J^*(\div\Phi) + \|\Phi\|_\infty \\
        \text{s.t.} \quad &  \Phi \in X = H_0(\div;\Omega) \cap C_0(\Omega;\R^d),
    \end{aligned}
    \quad \right\}
\end{equation}
and show that it is the pre-dual problem to \eqref{eq:P}. Note that, in general, \eqref{eq:primC} will not admit a solution.
We start with the Fenchel conjugate of $F: X \ni \Phi \mapsto \|\Phi\|_\infty \in \R$, which reads
\begin{equation*}
    F^*(\omega) = \sup_{\Phi \in X} \big( \dual{\omega}{\Phi}_{X^*, X} -  \|\Phi\|_\infty \big) 
    = \begin{cases}
        0, & \sup_{\Phi\in X, \|\Phi\|_\infty = 1} \dual{\omega}{\Phi}_{X^*, X}  \leq 1, \\
        \infty, & \text{else.}
    \end{cases}
\end{equation*}
Now, consider an element $\omega \in X^*$ with $F^*(\omega) < \infty$. Then the above implies for all $\Phi \in X$ that 
\begin{equation*}
    \dual{\omega}{\Phi}_{X^*, X}  \leq \|\Phi\|_\infty
\end{equation*}
and consequently, by the Hahn-Banach theorem, $\omega$ can be extended to a functional on $C_0(\Omega;\R^d)$.
Therefore, for every $\omega \in X^*$ with $F^*(\omega) < \infty$, there exists $\hat\omega \in  \frakM(\Omega;\R^d)$ such that 
\begin{equation}\label{eq:hatomega}
    \dual{\omega}{\Phi}_{X^*, X} = \int_\Omega \Phi \cdot \d \hat\omega \quad \forall \, \Phi \in X, \quad 
    |\hat\omega|(\Omega) = \|\hat\omega\|_{\frakM} \leq 1.
\end{equation}
On the other hand, let $\omega \in X^*$ be given such that \eqref{eq:hatomega} holds with some $\hat\omega \in  \frakM(\Omega;\R^d)$. Then it follows that 
\begin{equation*}
    \sup_{\Phi\in X, \|\Phi\|_\infty = 1} \dual{\omega}{\Phi}_{X^*, X} 
    =  \sup_{\Phi\in X, \|\Phi\|_\infty = 1} \int_\Omega \Phi \cdot \d \hat\omega \leq |\hat\omega|(\Omega) \leq 1
\end{equation*}
and thus $F^*(\omega) = 0 < \infty$. This shows that $F^*(\omega)$ is finite if and only if $\hat\omega \in  \frakM(\Omega;\R^d)$ exists such that \eqref{eq:hatomega}
holds and thus $F^*$ is the indicator functional of the 1-ball in $\frakM(\Omega;\R^d)$ as expected, i.e., 
\begin{equation}\label{eq:infnormdual}
    F^*(\omega) = 
    \begin{cases}
        0, & \exists\, \hat\omega \in  \frakM(\Omega;\R^d) \colon  \dual{\omega}{\Phi}_{X^*, X} = \int_\Omega \Phi \cdot \d \hat\omega \;\; \forall \, \Phi \in X, \;\;
        |\hat\omega|(\Omega) \leq 1,\\
        \infty, & \text{else.}
    \end{cases}
\end{equation}
Since $J^*$ is convex and finite, it is locally Lipschitz continuous everywhere. The continuity of $J^*$ and $F$ allows to apply the duality theorem from 
\cite[Theorem~III.4.1]{ET76} which gives
\begin{equation}\label{eq:fenchelrock}
    \inf_{\Phi\in X} \; F(\Phi) + J^*(\div\Phi) = \max_{u \in L^2(\Omega)} \big( - F^*(\nabla u) - J(u)\big),
\end{equation}
where $\nabla = (-\div)^* : L^2(\Omega) \to H_0(\div;\Omega)^* \subset X^* = H_0(\div;\Omega)^* + \frakM(\Omega;\R^d)$ denotes the distributional gradient.
Thanks to \eqref{eq:infnormdual}, the maximization problem in \eqref{eq:fenchelrock} is equivalent to 
\begin{equation}\label{eq:fenchel0}
    - \max_{u \in L^2(\Omega)} \big( - F^*(\nabla u) - J(u)\big) 
    = \min \{ J(u) \colon u \in L^2(\Omega), \; \nabla u \in \frakM(\Omega;\R^d), \; |\nabla u|(\Omega) \leq 1\},
\end{equation}
which is just the optimal control problem \eqref{eq:P}. 
Therefore, in view of \eqref{eq:fenchelrock}, \eqref{eq:primC} is indeed the pre-dual problem to \eqref{eq:P}. 
We point out that this is a classical observation and all but new, we only refer to \cite[Section~2.2]{CK11} for a similar setting.

Next we turn to the dual problem of \eqref{eq:Peps}. For this purpose, let us define the convex set 
\begin{equation*}
   \CC_\varepsilon := \Big\{\mu \in \HH^* \colon \dual{-\mu}{\varphi} - \frac{\varepsilon}{2}\, a[\varphi, \varphi] \leq 1
   \quad \forall\, \varphi \in C^\infty_c(\Omega;\R^d), \; \| \varphi\|_\infty \leq 1 \Big\}
\end{equation*}
and rewrite \eqref{eq:Peps} as 
\begin{equation*}
    \eqref{eq:Peps}
    \quad \Longleftrightarrow \quad 
    \min_{u\in L^2(\Omega)} \; J(u) + I_{\CC_\varepsilon}(\nabla u),
\end{equation*}
where $\nabla : L^2(\Omega) \to H_0(\div;\Omega)^* \embed \HH^*$ again denotes the distributional gradient. 
Moreover, $I_{\CC_\varepsilon}$ denotes the indicator functional of the set $\CC_\varepsilon$.
We prove that $u \equiv 0$ is a Slater point for this problem. 
To this end, let $\mu \in \HH^*$ with $\|\mu\|_{\HH^*} \leq \sqrt{2\varepsilon\beta}$ be arbitrary, where $\beta > 0$ denotes the coercivity constant of $a$, see \eqref{eq:acoer}.
Then, for all $\varphi \in \HH$, Young's inequality implies that
\begin{equation*}
    \dual{-\mu}{\varphi} - \frac{\varepsilon}{2}\, a[\varphi, \varphi]
    \leq \|\mu\|_{\HH^*} \|\varphi\|_{\HH} - \frac{\varepsilon}{2} \,\beta \, \|\varphi\|_{\HH}^2 
    \leq  1
\end{equation*}
and consequently $B_{\HH^*}(0,\sqrt{2\varepsilon\beta})\subset \CC_\varepsilon$. Hence $u\equiv 0$ is a Slater point for \eqref{eq:Peps}
and the Fenchel-Rockafellar duality, cf.\ e.g.\ \cite[Theorem~5.10]{CV20}, yields that 
\begin{equation}\label{eq:fencheleps}
     \min_{u\in L^2(\Omega)} \; J(u) + I_{\CC_\varepsilon}(\nabla u)
     = \max_{\Phi \in \HH} \; - J^*(\div \Phi) - I_{\CC_\varepsilon}^*(\Phi),
\end{equation}
where we identified $\HH^{**}$ with $\HH$ by means of the Riesz isomorphism. A precise characterization of the support functional $I_{\CC_\varepsilon}^*$ is rather intricate, 
but, fortunately, we do not need it for our convergence result, as the estimate derived in the following is sufficient for the upcoming analysis.
To this end, let $\Phi \in \HH \cap L^\infty(\Omega;\R^d)$, $\Phi \neq 0$, be arbitrary. Then Assumption~\ref{assu:density} implies that, 
for all $\mu \in \CC_\varepsilon$, there holds that
\begin{equation*}
    \dual{\mu}{\Phi}
    = \|\Phi\|_\infty \, \Big\langle - \mu, - \frac{\Phi}{\|\Phi\|_\infty} \Big\rangle 
    \leq \|\Phi\|_\infty +  \frac{\varepsilon}{2} \, \frac{1}{\|\Phi\|_\infty} \, a[\Phi, \Phi] 
\end{equation*}
and thus 
\begin{equation}\label{eq:indicatordual}
    I^*_{\CC_\varepsilon}(\Phi) 
    =  \sup_{\mu \in \CC_\varepsilon} \dual{\mu}{\Phi} \leq \|\Phi\|_\infty +  \frac{\varepsilon}{2} \, \frac{1}{\|\Phi\|_\infty} \, a[\Phi, \Phi] .
\end{equation}
Now, we have everything at hand to prove the following a priori estimate:

\begin{theorem}\label{thm:convrate}
    Let Assumption~\ref{assu:Lsmooth} be satisfied and assume that 
    \begin{equation}\label{eq:HHHdiv}
        \HH = H_0(\div;\Omega).
    \end{equation}        
    Denote the optimal solutions to \eqref{eq:P} and \eqref{eq:Peps} 
    by $u^*$ and $u_\varepsilon$, respectively, (which are unique as $\alpha > 0$). Then there is a constant $C > 0$ such that  
    \begin{equation}\label{eq:convrate}
        0 \leq J(u^*) - J(u_\varepsilon) \leq C \, \varepsilon \quad \text{and} \quad 
        \| u^* - u_\varepsilon \|_{L^2(\Omega)} \leq C\, \sqrt{\varepsilon}.
    \end{equation}
\end{theorem}

\begin{proof}
    First note that Assumption~\ref{assu:density} is satisfied because of \eqref{eq:HHHdiv}, cf.~Remark~\ref{rem:density}.
    We start with the estimate for the objective functional. Due to 
    \begin{equation*}
        \int_\Omega u^* \, \div \varphi \,\d x - \frac{\varepsilon}{2} \, a[\varphi, \varphi] 
        \leq \int_\Omega u^* \, \div \varphi \,\d x \leq \TV(u^*) \leq 1
        \quad \forall\, \varphi \in C^\infty_c(\Omega;\R^d) \colon \|\varphi\|_\infty \leq 1, 
    \end{equation*}
    the optimal solution $u^*$ is feasible for \eqref{eq:Peps} for every $\varepsilon > 0$. This yields 
    \begin{equation*}
    \begin{aligned}
        J(u_\varepsilon) 
        & \leq J(u^*)\\
        & = - \max_{u \in L^2(\Omega)} \big( - F^*(\nabla u) - J(u)\big)  & & \text{by \eqref{eq:fenchel0}}\\
        & = - \inf_{\Phi\in X} \; F(\Phi) + J^*(\div\Phi) & & \text{by \eqref{eq:fenchelrock}}\\
        & = \sup_{\Phi\in H_0(\div;\Omega) \cap C_0(\Omega;\R^d)} \; - J^*(\div\Phi)  - \|\Phi\|_\infty \\ 
        & \leq \sup_{\Phi\in H_0(\div) \cap L^\infty(\Omega;\R^d)} \; - J^*(\div\Phi)  - \|\Phi\|_\infty\\
        & \leq - J^*(\div \overline{\Phi} )  - \| \overline{\Phi} \|_\infty & & \text{by Lemma~\ref{lem:barPhi}}\\
        & \leq - J^*(\div\overline{\Phi})- I^*_{\CC_\varepsilon}(\overline\Phi) +  \frac{\varepsilon}{2} \, \frac{1}{\|\overline\Phi\|_\infty} \, a[\overline\Phi, \overline\Phi]
        & & \text{by \eqref{eq:indicatordual}} \\
        & \leq \max_{\Phi \in H_0(\div)} \; - J^*(\div \Phi) - I_{\CC_\varepsilon}^*(\Phi) +  \frac{\varepsilon}{2} \, \frac{1}{\|\overline\Phi\|_\infty} \, a[\overline\Phi, \overline\Phi] \\
        &  = \min_{u\in L^2(\Omega)} \; J(u) + I_{\CC_\varepsilon}(\nabla u)  +  \frac{\varepsilon}{2} \, \frac{1}{\|\overline\Phi\|_\infty} \, a[\overline\Phi, \overline\Phi] 
        & & \text{by \eqref{eq:fencheleps}}\\
        & = J(u_\varepsilon)  +  \frac{\varepsilon}{2} \, \frac{1}{\|\overline\Phi\|_\infty} \, a[\overline\Phi, \overline\Phi],
    \end{aligned}
    \end{equation*}
    which gives the first claim with $C := \frac{1}{2\|\overline\Phi\|_\infty} \, a[\overline\Phi, \overline\Phi]$.
    
    To show the second assertion, we apply Clarkson's inequality to obtain
    \begin{equation*}
        \alpha \, \Big\| \frac{u^* + u_\varepsilon}{2} - u_d \Big\|^2_{L^2(\Omega)} + 
        \alpha \, \Big\| \frac{u^* - u_\varepsilon}{2} \Big\|^2_{L^2(\Omega)} 
        \leq \frac{\alpha}{2} \Big( \| u^* - u_d \|_{L^2(\Omega)}^2 + \| u_\varepsilon - u_d\|_{L^2(\Omega)}^2 \Big).
    \end{equation*}
    For the ease of notation, let us abbreviate $y^* := S(u^* + f)$ and $y_\varepsilon := S(u_\varepsilon + f)$. 
    Then we deduce from the above and the convexity of $j$ that
    \begin{equation*}
    \begin{aligned}
        \alpha \, \Big\| \frac{u^* - u_\varepsilon}{2} \Big\|^2_{L^2(\Omega)} 
        & \leq j(y^*) + \frac{\alpha}{2} \, \| u^* - u_d \|_{L^2(\Omega)}^2 
        + j(y_\varepsilon) +  \frac{\alpha}{2} \, \| u_\varepsilon - u_d\|_{L^2(\Omega)}^2 \\
        & \quad - 2 \Bigg( j\Big(\frac{y^* + y_\varepsilon}{2} \Big) + \frac{\alpha}{2} \, \Big\| \frac{u^* + u_\varepsilon}{2} - u_d \Big\|^2_{L^2(\Omega)} \Bigg) .
    \end{aligned}
    \end{equation*}
    The linearity of $S$ implies $j(\frac{y^* + y_\varepsilon}{2}) = j(S(\frac{u^* + u_\varepsilon}{2} + f))$ and thus we arrive at
    \begin{equation*}
        \alpha \, \Big\| \frac{u^* - u_\varepsilon}{2} \Big\|^2_{L^2(\Omega)} 
        \leq J(u^*) + J(u_\varepsilon) - 2\,J\Big(\frac{u^* + u_\varepsilon}{2} \Big).
    \end{equation*}
    As seen at the beginning of the proof, $u^*$ is feasible for \eqref{eq:Peps} and, due to the convexity of $\TV_\varepsilon$ by Lemma~\ref{lem:L2TVeps}, 
    so is $\frac{u^* + u_\varepsilon}{2}$. The optimality of $u_\varepsilon$ thus gives
    \begin{equation*}
        \alpha \, \Big\| \frac{u^* - u_\varepsilon}{2} \Big\|^2_{L^2(\Omega)} \leq J(u^*) - J(u_\varepsilon),
    \end{equation*}
    which implies the second assertion with $C := \sqrt{\frac{2}{\alpha\|\bar\Phi\|_\infty}\, a[\bar\Phi, \bar\Phi]}$.
\end{proof}

\begin{remark}
    The inequality in the first part of the above proof can be slightly sharpened. To be more precise, although $L^\infty(\Omega;\R^d)$ is not separable, it holds that 
    \begin{equation*}
        \sup_{\Phi\in H_0(\div;\Omega) \cap C_0(\Omega;\R^d)} \; - J^*(\div\Phi)  - \|\Phi\|_\infty 
        = \sup_{\Phi\in H_0(\div) \cap L^\infty(\Omega;\R^d)} \; - J^*(\div\Phi)  - \|\Phi\|_\infty.
    \end{equation*}
    This can for instance be seen by employing the density result from \cite[Theorem~1]{HR15}. As this observation does not improve our 
    a priori estimate, we do not go into detail. 
\end{remark}

Let us conclude this section with a discussion on the assumption in \eqref{eq:HHHdiv}. An inspection of the proof of Theorem~\ref{thm:convrate} shows 
that one can choose a Hilbert space $\HH$ different from $H_0(\div;\Omega)$ and still obtain the same convergence rates, provided that there is a solution 
of the pre-dual problem \eqref{eq:priminf} that satisfies $\bar\Phi\in \HH$. This leads to the following

\begin{corollary}\label{cor:convrate}
    Let Assumptions~\ref{assu:density} and \ref{assu:Lsmooth} hold true and assume that $\HH$ is a proper subspace of $H_0(\div;\Omega)$, such as, e.g., $\HH = H^1_0(\Omega;\R^d)$. 
    Assume further that there is at least one solution of the pre-dual problem \eqref{eq:priminf} that satisfies $\bar\Phi \in \HH$. 
    Then the a priori estimates from \eqref{eq:convrate} are valid.
\end{corollary}

It is an open question to discuss sufficient conditions on the data such as $j$ and $\Omega$ that guarantee the additional regularity of $\bar\Phi$ required in the above corollary. 
This gives rise to future research.

\section{Outer approximation} \label{sec:Outer_Approximation}

Albeit regularized, problem \eqref{eq:Peps} is still not straight forward to solve as it involves uncountably many linear inequality constraints of the form
\begin{equation} \label{eq:linear_constraints}
	F_\eps(u, \phi) \leq 1 \quad \forall \, \phi \in C^\infty_c(\Omega;\R^d) \colon \| \phi \|_{\infty} \leq 1,
\end{equation}
where $F_\varepsilon$ is as defined in \eqref{eq:defFeps}. A common strategy to cope with constraints of this type is outer approximation, where only finitely many 
constraints of the form \eqref{eq:linear_constraints} are considered and infeasible solutions are cut off by iteratively adding new constraints to the problem. 
This method is well established for the solution of finite dimensional optimization problems, in particular 
mixed-integer nonlinear programs, see, e.g., the classical references \cite{DG86, FL94}, but has recently successfully been applied to infinite dimensional problems 
as well, cf.\ \cite{BGM24b, MS24}.

The construction of the algorithm as well as the associated convergence analysis heavily depends on the following

\begin{assumption}\label{assu:outerapprox}
    Throughout this section, the regularization parameter $\varepsilon > 0$ is fixed and 
    we assume that the Tikhonov parameter satisfies $\alpha > 0$ and that Assumption~\ref{assu:density} is fulfilled, i.e., 
    $\{ \varphi \in C^\infty_c(\Omega;\R^d) \colon \|\varphi\|_\infty \leq 1\}$ is dense in $\{ \varphi \in \HH \colon \|\varphi\|_\infty \leq 1\}$. 
\end{assumption}

Under this assumption, there exists a unique solution $u^* \in L^2(\Omega) \cap \BV(\Omega)$ of the original problem \eqref{eq:P} and, for every $\varepsilon > 0$, 
a unique solution $u_\varepsilon \in L^2(\Omega)$ of \eqref{eq:Peps}. Moreover, we have Lemma~\ref{lem:L2TVeps} at our disposal such that 
\begin{equation}\tag{Q$_\varepsilon$}\label{eq:Qeps}
    \TV_\eps(u) = \max \left\{F_\varepsilon(u, \varphi) : \varphi\in \HH, \; \|\varphi\|_\infty \leq 1 \right\},
\end{equation}        
where the maximization problem on the right hand side admits a unique solution. This maximizer is essential for the construction of the algorithm, which reads as follows:
We consider a relaxation of \eqref{eq:Peps} with $k \in \N_0$ constraints of the form \eqref{eq:linear_constraints}, which we denote by \eqref{eq:Pk}, 
and compute its minimizer $u_k \in L^2(\Omega)$. Subsequently, we compute $\TV_\eps(u_k)$ by solving the maximization problem \eqref{eq:Qeps} for $u = u_k$, 
which we denote by \eqref{eq:Qk} and which admits a unique maximizer $\phi_{k+1} \in \HH$. If $\TV_\eps(u_k) = F_\eps(u_k,\phi_{k+1}) \leq 1$,  
we have found an optimal solution to \eqref{eq:Peps}, because $u_k$ is optimal for the relaxation \eqref{eq:Pk} of \eqref{eq:Peps} and feasible for \eqref{eq:Peps}. 
If $\TV_\eps(u_k) = F_\eps(u_k,\phi_{k+1}) >1$, we add the linear constraint $F_\eps(u,\phi_{k+1}) \leq 1$ to \eqref{eq:Pk} 
and denote the resulting problem by (P$_{k+1}$). 
This completes the iteration and we start the next one by solving (P$_{k+1}$). We provide the complete algorithm in Algorithm~\ref{alg:Reps}.

\begin{algorithm}[h]
	\caption{Outer approximation algorithm for \eqref{eq:Peps} \label{alg:Reps}}
	\begin{algorithmic}[1]
		\State Set $k=0$
		\State Solve 
		\begin{equation}\tag{P$_k$}\label{eq:Pk}
			\begin{aligned}
				\min_{u\in L^2(\Omega)} \quad & J(u) \\ 
				\text{s.t.} \quad & F_\varepsilon(u, \varphi_i) \leq 1 \quad \forall\,  i \in \N \text{ with } i \leq k
			\end{aligned}
		\end{equation} 
		and denote the optimal solution by $u_k$. \label{algReps2}
		\State Compute $\TV_\eps(u_k)$ by solving
		\begin{equation}\tag{Q$_k$} \label{eq:Qk}
			\begin{aligned}
				\max_{\varphi \in \HH} \quad & F_\varepsilon(u_k, \varphi)\\
				\text{s.t.} \quad &  \|\varphi\|_\infty \leq 1
			\end{aligned}
		\end{equation}           			
		and denote the optimal solution by $\varphi_{k+1}$.\label{algReps3}
		\If{$F_\varepsilon(u_k, \varphi_{k+1}) \leq 1$}
		\State \Return $u_k$ as optimal solution to \eqref{eq:Peps}.
		\EndIf
		\State Set $k \leftarrow k+1$ and go to Step~\ref{algReps2}.
	\end{algorithmic}
\end{algorithm}

\begin{lemma}\label{lem:barufeas}
	Let Assumption~\ref{assu:outerapprox} be fulfilled and denote by $\{u_k\}_{k \in \N}$ the sequence of optimal solutions to \eqref{eq:Pk} produced by \cref{alg:Reps}. 
    Then there exists a weakly converging subsequence and every such subsequence of $\{u_k\}_{k \in \N}$ also converges strongly in $L^2(\Omega)$. 
    For each accumulation point $\bar{u} \in L^2(\Omega)$ of $\{u_k\}_{k \in \N}$, there holds $\TV_\eps(\bar u) \leq 1$, that is, $\bar{u}$ is feasible for \eqref{eq:Peps}.
\end{lemma}   
    			
\begin{proof}
    First observe that $u \equiv 0$ is feasible for every \eqref{eq:Pk}, $k\in \N$, due to the coercivity of the bilinear form $a$ and therefore, the boundedness of $\{u_k\}$ 
    in $L^2(\Omega)$ follows exactly as in \eqref{eq:ubound}. This gives the existence of a weak accumulation point. 
	
	Next we prove that each weakly converging subsequence of the sequence $\{u_k\}_{k \in \N}$ actually converges strongly in $L^2(\Omega)$. 
	To this end, we consider a subsequence of $\{u_k\}_{k \in \N}$ which we denote by the same symbol for simplicity and that converges weakly 
	to some $\bar{u} \in L^2(\Omega)$, that is, $u_k \rightharpoonup \bar{u}$ in $L^2(\Omega)$. 
    Now let $i \in \N$ be fixed but arbitrary. Then, for every index $k$ of the subsequence large enough, the constraints in \eqref{eq:Pk} imply
	\begin{align*}
		F_\eps(u_k, \varphi_i) = - \frac{\eps}{2} a[\phi_i,\phi_i] + \int_\Omega u_k \div \phi_i \d x \leq 1
	\end{align*}
	and, by passing to the limit $k \to \infty$, we obtain
	\begin{equation}\label{eq:ineqi}
		F_\eps(\bar{u},\varphi_i) = - \frac{\eps}{2} a[\phi_i,\phi_i] + \int_\Omega \bar{u} \div \phi_i \d x \leq 1.
	\end{equation}
	Since $i\in \N$ was arbitrary, this holds for all $i\in \N$ such that $\bar{u}$ is feasible for every problem \eqref{eq:Pk}, $k \in \N$. 
	Since $u_k$ is optimal for \eqref{eq:Pk}, there holds $J(u_k) \leq J(\bar{u})$ for all $k \in \N$. Due to the weak lower semicontinuity of $J$, we thus obtain
	\begin{equation*}
		J(\bar{u}) \leq \liminf_{k \to \infty} J(u_k) \leq \limsup_{k \to \infty} J(u_k) \leq J(\bar{u})
	\end{equation*}
	and therefore $J(u_k) \to J(\bar{u})$ as $k \to \infty$. 
    Therefore, since both components of the objective, i.e., $u \mapsto j(S (u_\varepsilon + f))$ and $u\mapsto \frac{\alpha}{2}\, \| u - u_d \|_{L^2(\Omega)}^2 $, 
    are weakly lower semicontinuous, they converge separately, too, i.e., in particular
    \begin{equation*}
        \frac{\alpha}{2}\,  \| u_\varepsilon - u_d \|_{L^2(\Omega)}^2 \to \frac{\alpha}{2}\, \| u^* - u_d \|_{L^2(\Omega)}^2, 
    \end{equation*}
    which in turn implies that $\|u_\varepsilon\|_{L^2(\Omega)}$ converges to  $\|u^*\|_{L^2(\Omega)}$. 
    Together with the weak convergence, this yields strong convergence of the subsequence, i.e., $u_k \to \bar{u}$ in $L^2(\Omega)$ as $k \to \infty$.
	
	Next we prove the feasibility of $\bar{u}$ for \eqref{eq:Peps}, that is, we prove $\TV_\eps(\bar{u}) \leq 1$. To this end, we first note that 
    the sequence $\{\varphi_{k+1}\}$ is bounded in $\HH$, which is seen as follows: The optimality of $\varphi_{k+1}$ implies
    \begin{equation*}
        F_\eps(u_{k},\phi_{k+1}) \geq F_\eps(u_k, 0) = 0
        \quad \Longrightarrow \quad \frac{\varepsilon}{2} \, a[\varphi_{k+1}, \varphi_{k+1}] \leq \int_\Omega u_k \, \div \varphi_{k+1} \,\d x,
    \end{equation*}
    such that the boundedness of $\|u_k \|_{L^2(\Omega)}$ along with the continuous embedding of $\HH$ in $H_0(\div;\Omega)$ and the coercivity of $a$ 
    gives the desired boundedness of $\{\varphi_k\}$ in $\HH$.
	Therefore, by restricting to a further subsequence, if necessary, we obtain sequences, denoted again by the same symbol, 
	such that $u_k \to \bar{u}$ and $\phi_{k+1} \weakly \tilde{\phi}$ in $\HH$.
	Since, the set $\{ \phi \in \HH : \| \phi \|_\infty \leq 1 \}$ is non-empty, closed, and convex and therefore weakly closed, 
	$\tilde{\phi}$ is feasible for \eqref{eq:Qeps}. Denote the optimal solution to \eqref{eq:Qeps} with $u = \bar{u}$ by $\bar{\varphi}$ 
	so that $\TV_\varepsilon(\bar u) = F_\varepsilon(\bar u, \bar\varphi)$. 
    Then, together with the weak lower semicontinuity of $\HH \ni \varphi \mapsto a[\varphi, \varphi] \in \R$ and the strong and weak convergence of $\{u_k\}$ and 
    $\{\varphi_{k+1}\}$, respectively,  
    the optimality of $\bar\varphi$ and the feasibility of $\tilde\varphi$ imply	
	\begin{align*}
		F_\eps(\bar{u}, \bar{\phi}) \geq F_\eps(\bar{u}, \tilde{\phi}) 
		& = - \frac{\eps}{2}\, a[\tilde{\phi}, \tilde{\phi}] + \int_\Omega \bar{u} \div \tilde{\phi} \dd x \\
		& \geq \limsup_{k \to \infty} \left( - \frac{\eps}{2}\, a [\phi_{k+1},\phi_{k+1}] \right) + \lim_{k \to \infty} \int_\Omega u_{k} \div \phi_{k+1} \dd x \\
		& \geq \limsup_{k \to \infty} F_\eps(u_k,\phi_{k+1}) \\
		& \geq \liminf_{k \to \infty} F_\eps(u_k,\phi_{k+1}) \\
		& \geq \liminf_{k \to \infty} F_\eps(u_k,\bar{ \phi}) 
		= F_\eps(\bar{u}, \bar{\phi}),
	\end{align*}
	where the last inequality holds because $\phi_{k+1}$ is the optimal solution to \eqref{eq:Qk}, which yields $F_\eps(u_k, \varphi_{k+1}) \geq F_\eps(u_k,\bar{\varphi})$
	for all $k\in \N$. This implies
	\begin{align}\label{eq:convFeps}
		F_\varepsilon(u_k, \varphi_{k+1}) \to F_\varepsilon(\bar u, \tilde \varphi)=F_\varepsilon(\bar u, \bar\varphi)
	\end{align}
	as $k \to \infty$.
	Since \eqref{eq:Qeps} is uniquely solvable for every given $u \in L^2(\Omega)$ according to Lemma~\ref{lem:L2TVeps}, 
	this gives $\bar\varphi = \tilde\varphi$, as $\bar{ \phi}$ is optimal and $\tilde{\phi}$ is feasible for \eqref{eq:Qeps} with $u = \bar u$.
	Hence, $\varphi_{k+1} \weakly \bar\varphi$ in $\HH$ and, along with the strong convergence of $\{u_k\}$, we deduce from \eqref{eq:convFeps} that
	\begin{align*}
		\frac{\varepsilon}{2}\,a[\bar\varphi,\bar{\varphi}]
		&\leq \liminf_{k \to \infty} \frac{\varepsilon}{2}\,a[\varphi_{k+1},\varphi_{k+1}]\\
		&\leq \limsup_{k \to \infty} \frac{\varepsilon}{2}\,a[\varphi_{k+1},\varphi_{k+1}]\\
		&\leq \limsup_{k \to\infty} - F_\varepsilon(u_k, \varphi_{k+1})
		+ \limsup_{k \to \infty} \int_\Omega u_k \,\div \varphi_{k+1} \, \d x\\
		&= - F_\varepsilon(\bar u, \bar\varphi) + \int_\Omega \bar u\,\div \bar\varphi \, \d x 
		= \frac{\varepsilon}{2}\,a[\bar\varphi,\bar{\varphi}]
	\end{align*}
	and therefore $a[\varphi_{k+1},\varphi_{k+1}] \to a[\bar\varphi,\bar{\varphi}]$ as $k \to \infty$. We now return to \eqref{eq:ineqi}, 
	which holds for every $i\in \N$ and thus also for $i = k+1$, i.e.,  
	\begin{equation*}
	    1 \geq F_\eps(\bar{u},\varphi_{k+1}) = - \frac{\eps}{2} a[\phi_{k+1},\phi_{k+1}] + \int_\Omega \bar{u} \div \phi_{k+1} \d x 
	    \to  - \frac{\eps}{2} a[\bar\phi, \bar\phi] + \int_\Omega \bar{u}\, \div \bar\phi\, \d x = \TV_\varepsilon(\bar u). 
	\end{equation*}
	This proves that every accumulation point of $\{u_k\}_{k \in \N}$ is feasible for \eqref{eq:Peps}.
\end{proof}

\begin{theorem}\label{thm:conveps}
	Let Assumption~\ref{assu:outerapprox} hold true and assume that \Cref{alg:Reps} does not stop after a finite number of iterations. 
	Then the whole sequence $\{u_k\}_{k \in \N}$ generated by \Cref{alg:Reps} converges strongly in $L^2(\Omega)$ to the unique minimizer of \eqref{eq:Peps}.
\end{theorem}

\begin{proof}
    From Lemma~\ref{lem:barufeas} we know that the sequence $\{u_k\}$ admits a strong accumulation point $\bar u$ in $L^2(\Omega)$, which is feasible for \eqref{eq:Peps}.
    To show its optimality, consider an arbitrary $u \in L^2(\Omega)$ that is feasible for \eqref{eq:Peps}. 
	Then, by construction, it is also feasible for \eqref{eq:Pk} for each $k\in \N$ and the optimality of $u_k$ implies
	$J(u_k) \leq J(u)$. The continuity of $J$ thus gives
	\begin{equation*}
		J(\bar u) \leq \lim_{k\to\infty} J(u_k) \leq J(u)
	\end{equation*}
	and, since $u$ was an arbitrary feasible point for \eqref{eq:Peps}, this yields the optimality of $\bar u$ for \eqref{eq:Peps}.
	From Proposition~\ref{prop:mitTikhonov}, we know that \eqref{eq:Peps} is uniquely solvable  and thus $\bar u$ equals the unique solution $u_\varepsilon$ of \eqref{eq:Peps}.
    By the Urysohn subsequence principle, the whole sequence $\{u_k\}_{k \in \N}$ therefore converges strongly to the unique minimizer $u_\varepsilon$.
\end{proof}

\begin{remark}
    Let us assess the effort of Algorithm~\ref{alg:Reps}. In each iteration, one has to solve \eqref{eq:Pk}, which is a ``classical'' elliptic optimal control problem 
    with finitely many linear integral control constraints, and \eqref{eq:Qk}, which is an obstacle-type problem in the Hilbert space $\HH$.
    This means, in each iteration, two more or less standard problems have to be solved. We will come back to this point in Section~\ref{sec:semismooth}.
\end{remark}

\section{Optimality conditions for \eqref{eq:P}} \label{sec:optimality}

For the construction of an exact solution, we derive necessary and sufficient optimality conditions to \eqref{eq:P} following the lines of \cite{CK19}, 
where a similar problem is analyzed. For that purpose we rewrite \eqref{eq:P} as
\begin{equation}
    \eqref{eq:P} 
    \quad \Longleftrightarrow \quad    
	\min_{u \in \UU} \; J(u) + I_\BB(\nabla u) 
\end{equation}
where $\UU$ is as defined in \eqref{eq:defU}, $\nabla \in \LL(\BV(\Omega);\frakM(\Omega;\R^d))$ denotes the distributional gradient, and
\begin{equation*}
	\BB \coloneqq \{ \mu \in \frakM(\Omega; \R^d) \colon |\mu|(\Omega) \leq 1 \}.
\end{equation*}
Moreover, $I_\BB: \frakM(\Omega;\R^d)$ again denotes the indicator functional.
Note again that, due to $d\leq 4$, $L^{d/(d-1)}(\Omega) \embed H^{-1}(\Omega)$ such that $S$ and thus also $J$ are well defined on $\BV(\Omega)\embed \UU$.

In order to derive necessary and sufficient optimality conditions for \eqref{eq:P}, we follow the lines of \cite[Theorems 2 and 3]{CK19}.

\begin{theorem}\label{thm:noc}
    In addition to our standing assumptions, let $j$ be Fr\'echet-differentiable from $H^1_0(\Omega)$ to $\R$. Then $\bar u \in \UU$ is a minimizer of \eqref{eq:P}, 
    if and only if, there exist $\bar y \in H^1_0(\Omega)$, $\bar p \in H^1_0(\Omega)$ and $\bar\lambda  \in \frakM(\Omega;\R^d)^*$ such that the following optimality system is fulfilled:
	\begin{subequations}\label{eq:optimality}
		\begin{gather}
            - \laplace \bar y = \bar u + f \;\; \text{in } H^{-1}(\Omega), \quad 		
            - \laplace \bar p =  j'(\bar y)\;\; \text{in } H^{-1}(\Omega), \label{eq:adjpde}\\
            \nabla \bar u \in \BB,\quad \dual{\bar{\lambda}}{\mu - \nabla \bar{u}}_{\frakM^*, \frakM} \leq 0 \quad \forall \, \mu \in \BB, \label{eq:compl} \\
			\dual{\bar{\lambda}}{\nabla v}_{\frakM^*,\frakM} 
			+ \int_{\Omega}(\bar{p} + \alpha (\bar{u} - u_d) ) v \dd x = 0 \quad \forall \, v \in \UU. \label{eq:gradeq} 
		\end{gather}
	\end{subequations}
	If $\alpha = 0$ or (inclusive) $d\leq 2$, then \eqref{eq:compl} and \eqref{eq:gradeq} are equivalent to the existence of a function $\bar\Phi \in C_0(\Omega;\R^d)$ such that 
	\begin{subequations}\label{eq:optimalityC}
		\begin{gather}
            \nabla \bar u \in \BB,\quad \dual{\nabla v - \nabla \bar{u}}{\bar\Phi}_{\frakM, C_0} \leq 0 \quad \forall \, v \in \BV(\Omega) \colon \nabla v \in \BB,
            \tag{\ref{eq:compl}'} \label{eq:comp2l} \\
			\dual{\nabla v}{\bar\Phi}_{\frakM, C_0} 
			+ \int_{\Omega}(\bar{p} + \alpha (\bar{u} - u_d) ) v \dd x = 0 \quad \forall \, v \in \BV(\Omega). 
			\tag{\ref{eq:gradeq}'} \label{eq:gradeq2} 
		\end{gather}
	\end{subequations}
\end{theorem}

\begin{proof}
    Since \eqref{eq:P} is a convex problem, $\bar u$ is a minimizer of \eqref{eq:P}, if and only if, 
    \begin{equation}\label{eq:fermat1}
        0 \in \partial (J + I_{\BB}\circ \nabla)(\bar u).
    \end{equation}
    By the chain rule $J$ is Fr\'echet-differentiable from $\UU$ to $\R$ with derivative 
    \begin{equation*}
        J'(u) = S^* j'(S(u + f)) + \alpha \, u.
    \end{equation*}
    Since $u \equiv 0 $ is a Slater point for \eqref{eq:P}, i.e., $\nabla u \in \interior(\BB) = \interior(\dom(I_{\BB}))$, we are allowed to apply the sum and chain rule 
    for convex subdifferentials, which implies that \eqref{eq:fermat1} is equivalent to
    \begin{equation}\label{eq:fermat2}
        0 \in  S^* j'(S(\bar u + f)) + \alpha \, \bar u + \nabla^* \partial I_{\BB}(\nabla \bar u).
    \end{equation}
    If we define $\bar y \coloneqq S(\bar u + f)$ and $\bar p \coloneqq S^* j'(\bar y) \in H^1_0(\Omega)$, then standard adjoint calculus gives that $\bar p$ solves 
    the second Poisson equation in \eqref{eq:adjpde}. With this definition of $\bar p$, \eqref{eq:fermat2} is equivalent to the existence of $\bar\lambda \in \frakM(\Omega;\R^d)^*$
    such that \eqref{eq:gradeq} holds and $\bar\lambda \in \partial I_{\BB}(\nabla \bar u)$. The latter is equivalent to the variational inequality in \eqref{eq:compl}.
    All in all, we have seen that the optimality system \eqref{eq:optimality} is equivalent to \eqref{eq:fermat1}, which implies its necessity and sufficiency for optimality.
	
    The second assertion concerning the existence of a continuous function $\bar\Phi$ can be proven completely along the lines of \cite[Theorem~3]{CK19}. 
    For convenience of the reader, we shortly sketch the arguments. In order to show that $\bar\lambda$ can be identified with an element of the pre-dual space 
    $C_0(\Omega;\R^d)$, if $\alpha = 0$ or $d\leq 2$, we define the linear functional $\ell$ on the subspace 
    $V \coloneqq \{\nabla v \colon v\in \BV(\Omega)\}$ of $\frakM(\Omega;\R^d)$ as follows
    \begin{equation*}
        \ell : V \to \R, \quad \ell(\nabla v) = \dual{\bar\lambda}{\nabla v}_{\frakM^*, \frakM} .
    \end{equation*}
    We show that $\ell$ is continuous w.r.t.\ weak-$\ast$ convergence in $\frakM(\Omega;\R^d)$. This will in turn imply that it can be extended to a 
    weakly-$\ast$ continuous functional on the whole space $\frakM(\Omega;\R^d)$, which, by weak-$\ast$ continuity, can then be identified with an 
    element of the pre-dual space.
	In order to show that $\ell$ is weakly-$\ast$ continuous on its domain $V$, we prove that $\ell^{-1}(0)$ is weakly-$\ast$ 	sequentially closed.
	For this purpose, let $\{\mu_k\} \subset \ell^{-1}(0) \subset V$ be arbitrary with $\mu_k \weak^* \mu$ in $\frakM(\Omega;\R^d)$. 
    Then, for all $k\in \N$, the definition of $V$ gives the existence of $v_k \in \BV(\Omega)$ such that $\mu_k = \nabla v_k$ and thus $\mu_k = \nabla v_k^0$. 
    Since 
    \begin{equation*}
        \BV(\Omega) \ni v \mapsto \frac{1}{|\Omega|} \Big| \int_\Omega v\,\d x \Big| + \TV(v) \in \R
    \end{equation*}
    defines a norm on $\BV(\Omega)$, $\{v_k^0\}$ is bounded in $\BV(\Omega) \embed L^q(\Omega)$,  $q = \frac{d}{d-1}$, cf.\ Remark~\ref{rem:limitcase}.
    Thus there exists a subsequence, denoted by the same symbol to ease notation, and a function $v^0\in \BV(\Omega)$ such that 
    \begin{equation}\label{eq:weakv0}
        \nabla v_k^0 \weak^* \nabla v^0 \;\; \text{in } \frakM(\Omega;\R^d), \quad 
        v_k^0 \weak v^0 \;\; \text{in } L^{q}(\Omega)
        \quad \text{as } k \to \infty.
    \end{equation}
    As the weak-$\ast$ limit is unique, we find $\mu = \nabla v_0$ so that $\mu \in V$. Moreover, since $\bar p \in H^1_0(\Omega)$, Sobolev embeddings imply that
    \begin{equation*}
        \bar{p} + \alpha (\bar{u} - u_d) \in 
        \begin{cases}
            L^2(\Omega), & \text{if } \alpha > 0,\\
            L^r(\Omega), & \text{in } \alpha = 0,
        \end{cases}
        \quad \text{with} \quad 
        r = 
        \begin{cases}
            \frac{2d}{d-2}, & \text{if } d > 2,\\
            \in [1,\infty), & \text{if } d = 2,\\
            \infty, & \text{if } d = 1,
        \end{cases}
    \end{equation*}
    and therefore $\bar p + \alpha (\bar{u} - u_d) \in  L^d(\Omega) = L^{q'}(\Omega)$, if $\alpha = 0$ and $d \leq 4$ or $\alpha > 0$ and $d\leq 2$. 
    (Note that $d\leq 4$ is a standing assumption anyway.) Together with the weak convergence in \eqref{eq:weakv0}, this yields
    \begin{equation*}
    \begin{aligned}
        0 = \ell(\mu_k) = \dual{\bar\lambda}{\nabla v_k}_{\frakM^*, \frakM} 
        & = - \int_{\Omega}(\bar{p} + \alpha (\bar{u} - u_d) ) v_k^0 \dd x \\
        & \to - \int_{\Omega}(\bar{p} + \alpha (\bar{u} - u_d) ) v^0 \dd x = \dual{\bar\lambda}{\nabla v^0}_{\frakM^*, \frakM} = \ell(\mu).
    \end{aligned}
    \end{equation*}        
    This shows $\mu \in \ell^{-1}(0)$ and proves the weak-$\ast$ sequential closedness of $\ell^{-1}(0)$. 
    From \cite[Theorem~5.1]{Con90} and \cite[Theorem~1.6.8]{Schi07}, it then follows that $\ell^{-1}(0)$ is even weak-$\ast$ closed, giving in turn that $\ell$ is weak-$\ast$ continuous. 
    The Hahn-Banach theorem then implies that there is a weak-$\ast$ continuous linear extension on the whole space $\frakM(\Omega;\R^d)$, 
    cf., e.g., \cite[Theorem~3.6]{Rud91}.
    Since it is weak-$\ast$ continuous, this extension can be identified with a function $\bar\Phi$ from the pre-dual space $C_0(\Omega;\R^d)$, 
    see \cite[Proposition~3.14]{Bre10}.
    In this way, we have proven the existence of a function $\bar\Phi\in C_0(\Omega;\R^d)$ such that 
    \begin{equation*}
        \dual{\mu}{\bar\Phi}_{\frakM, C_0} = \ell(\mu) = \dual{\bar\lambda}{\mu}_{\frakM^*, \frakM} \quad \forall\, \mu \in V. 
    \end{equation*}
	Using this identity, we deduce \eqref{eq:comp2l}--\eqref{eq:gradeq2} from \eqref{eq:compl}--\eqref{eq:gradeq}. Note in this context that, 
	if $d\leq 2$, then $\BV(\Omega) \embed L^2(\Omega)$ so that, in this case, $\UU = \BV(\Omega)$ even if $\alpha > 0$.
	
	To show the reverse implication, let us assume that $\bar\Phi \in C_0(\Omega;\R^d)$ satisfies \eqref{eq:comp2l}--\eqref{eq:gradeq2}. 
	This time, we define the following linear mapping:
	\begin{equation*}
	    \ell : V \to \R, \quad \ell(\nabla v) = \dual{\nabla v}{\bar\Phi}_{\frakM, C_0}.
	\end{equation*}
	Then \eqref{eq:comp2l} implies that 
	\begin{equation*}
	    \ell(\mu) \leq \dual{\nabla \bar u}{\bar\Phi}_{\frakM, C_0} \, \|\mu\|_{\frakM(\Omega;\R^d)} \quad \forall\, \mu \in V.
	\end{equation*}
    and thus, the Hahn-Banach theorem (this time applied in the strong topology) gives the existence of 
    a linear and bounded functional $\bar\lambda \in \frakM(\Omega;\R^d)$ such that 
    \begin{align}
        & \dual{\bar\lambda}{\mu}_{\frakM^*, \frakM} = \dual{\mu}{\bar\Phi}_{\frakM, C_0} \quad \forall\, \mu \in V \label{eq:hahnbanach1}\\
        \quad \text{and} \quad
        & \dual{\bar\lambda}{\mu}_{\frakM^*, \frakM} \leq \dual{\nabla \bar u}{\bar\Phi}_{\frakM, C_0} \, \|\mu\|_{\frakM(\Omega;\R^d)} \quad \forall\, \mu \in \frakM(\Omega;\R^d) .
        \label{eq:hahnbanach2}
    \end{align}
    Inserting \eqref{eq:hahnbanach1} in \eqref{eq:gradeq2} implies \eqref{eq:gradeq}.
    Finally, \eqref{eq:compl} directly follows from \eqref{eq:hahnbanach2} and \eqref{eq:hahnbanach1} by taking into account that $\|\mu\|_{\frakM(\Omega;\R^d)} = |\mu|(\Omega)$.	
\end{proof}

\subsection{Construction of an exact optimal solution} \label{sec:exact_solution}

In order to test our algorithmic approach, we construct an exact solution based on the optimality system in Theorem~\ref{thm:noc}. 
We choose a two-dimensional domain, namely $\Omega \coloneqq (0,1)^2$, and the usual $L^2$-tracking-type objective for $j$, i.e.,
\begin{equation}\label{eq:tracking}
    j(y) \coloneqq \frac{1}{2}\, \| y - y_d \|_{L^2(\Omega)}^2 
\end{equation}
with a given desired state $y_d \in L^2(\Omega)$. Together with the additional inhomogeneity $f$ and the desired control $u_d$, the desired state $y_d$ 
allows us to construct an exact solution of \eqref{eq:P}, based on the necessary and sufficient optimality conditions from Theorem~\ref{thm:noc}, more precisely 
\eqref{eq:adjpde}, \eqref{eq:comp2l}, and \eqref{eq:gradeq2}. These read in case of \eqref{eq:tracking} as follows:
$\bar u \in \UU$ is optimal, iff there exist $\bar y, \bar p \in H^1_0(\Omega)$, and $\bar \Phi \in C_0(\Omega;\R^d)$ such that
\begin{subequations}\label{eq:nocex}
    \begin{gather}
        - \laplace \bar y = \bar u + f \;\; \text{in } H^{-1}(\Omega), \quad 		
        - \laplace \bar p =  \bar y - y_d \;\; \text{in } H^{-1}(\Omega), \label{eq:adjex}\\
        \TV(\bar u) \leq 1, \quad \dual{\nabla v - \nabla \bar{u}}{\bar\Phi}_{\frakM, C_0} \leq 0 \quad \forall \, v \in \BV(\Omega) \colon \TV(v) \leq 1, \label{eq:complex} \\
        \dual{\nabla v}{\bar\Phi}_{\frakM, C_0} 
        + \int_{\Omega}(\bar{p} + \alpha (\bar{u} - u_d) ) v \dd x = 0 \quad \forall \, v \in \BV(\Omega). \label{eq:gradeqex} 
    \end{gather}
\end{subequations}
%
We begin the construction of a solution of this system defining an optimal control with minimal regularity.
For this purpose, let us define $B \coloneqq B_r((0.5,0.5)^T) \subset \Omega$ as the ball around $(0.5,0.5)^T$ with radius $r \in \R_{>0}$ 
and perimeter $P:=P(B)=2 \pi r$. We then set 
\begin{equation*}
	\bar{u}(x) \coloneqq 
	\begin{cases}
		\frac{1}{P}, & \text{if } x \in B, \\
		0, & \text{if } x \in \Omega \setminus B.
	\end{cases}
\end{equation*}
We then have $\nabla \bar{u} = - \frac{1}{P}\, \nu_B\,  \HH^{1} \mres \partial B$, where $\nu_B$ denotes the outer unit normal to $B$ and $\HH^1$ is the one-dimensional 
Hausdorff measure. Note that $\TV(\bar{u}) \leq 1$. 
Next we construct a multiplier $\bar\Phi \in C_0(\Omega;\R^d)$ such that the complementarity relation in \eqref{eq:complex} is fulfilled.
To this end, assume for a moment that $\bar\Phi$ is chosen such that $\bar\Phi |_{\partial B} = - s\, \nu_B$ with $s \in \R_{>0}$. Then we obtain for every $\mu \in \frakM(\Omega;\R^d)$ 
with $|\mu|(\Omega) \leq 1$ that  
\begin{equation*}
	\langle \bar\Phi, \mu - \nabla \bar{u} \rangle_{C_0, \frakM} 
	= \int_\Omega \bar\Phi \dd \mu + \frac{1}{P} \int_{\partial B} \bar\Phi \cdot \nu_B \dd \HH^1 
	\leq \| \bar\Phi \|_\infty \, \underbrace{| \mu | (\Omega)}_{\leq 1} - \frac{s}{P} \, \HH^1 (\partial B) \leq \| \bar\Phi \|_\infty - s.
\end{equation*}
Hence, if we choose $\bar\Phi$ such that $\| \bar\Phi \|_\infty \leq s$, then \eqref{eq:complex} is met.
To this end, we set $r = \frac{1}{4}$ and 
\begin{equation}\label{eq:Phiex}
\begin{aligned}
    \bar\Phi(x) \coloneqq \Phi_0\big( x-(0.5,0.5)^\top \big)
    \quad & \text{with} & & 
	\Phi_0 (x) \coloneqq 
	\begin{cases}
		- s \, \Psi(\| x \|_2) \frac{x}{\| x \|_2}, & \text{if } \| x \|_2 \in (\frac{3}{16},\frac{5}{16}),\\
		0, & \text{else}
	\end{cases}\\
	& \text{and} & &
    \Psi \in C^1_c((0,\tfrac{1}{2})), \;\; \Psi(\tfrac{1}{4}) = 1, \;\; | \Psi(r) | \leq 1\;\; \forall \,r \in (0,\tfrac{1}{2}).
\end{aligned}
\end{equation}
For example, we may choose 
\begin{equation*}
	\Psi(r) \coloneqq \begin{cases}
		-8192\, r^3 + 5376\, r^2 - 1152\, r + 81, & \text{if } r \in [\frac{3}{16},\frac{1}{4}], \\
		8192\, r^3 - 6912 \, r^2 + 1920 \, r - 175, & \text{if } r \in (\frac{1}{4},\frac{5}{16}] \\
		0, & \text{if } r \in (0,\frac{1}{2}) \setminus [\frac{3}{16},\frac{5}{16}].
	\end{cases}
\end{equation*}
Note that, by construction, $\bar\Phi \in C^1_c(\Omega;\R^d)$ so that \eqref{eq:gradeqex} is equivalent to
\begin{equation}\label{eq:gradex2}\tag{\ref{eq:gradeqex}'}
    - \div \bar\Phi + \bar p + \alpha(\bar u - u_d) = 0 \quad \text{a.e.\ in } \Omega.
\end{equation}
For the optimal state and the adjoint state, we choose $\bar{y}(x_1,x_2) = \bar{p}(x_1,x_2) \coloneqq 0.1\, \sin (2 \pi\, x_1) \sin (2 \pi\, x_2)$ 
so that the homogeneous Dirichlet conditions are fulfilled. To fulfill the PDEs in \eqref{eq:adjex}, $f$ and $y_d$ are set to 
$f \coloneqq - \Delta \bar{y} - \bar{u}$ and $y_d \coloneqq \Delta \bar{p} + \bar{y}$. Note that $\bar y, \bar p \in H^2(\Omega)$ such that 
$f, y_d\in L^2(\Omega)$.
Finally, in order to satisfy \eqref{eq:gradex2}, we set $u_d \coloneqq \bar{u} + \frac{1}{\alpha} (\bar{p} - \div \bar\Phi) \in L^2(\Omega)$.

For our numerical experiments in the next section, $s$ from \eqref{eq:Phiex} is set to $s = 0.01$ and we choose $\alpha=1$ for the Tikhonov parameter.

\section{Numerical examples} \label{sec:numerics}

\subsection{Numerical realization}\label{sec:semismooth}

In our numerical example, we choose the setting from Section~\ref{sec:exact_solution}, i.e., the domain is set to  
$\Omega = (0,1)^2$ and we choose the tracking type objective from \eqref{eq:tracking}. For the Hilbert space $\HH$, we choose $\HH = H^1_0(\Omega;\R^2)$ and 
the bilinear form $a: \HH \times \HH \to \R$ is set to 
\begin{align*}
	a[\varphi,\varphi] = \int_\Omega \nabla^{\textup s}\varphi : \C \, \nabla^{\textup s}\varphi \dd x,
\end{align*}
where $\nabla^{\textup s}\varphi = \frac{1}{2}(\nabla \varphi + \nabla \varphi^\top)$ is the symmetrized Jacobian and $\C \in \LL(\R^{d \times d}_{\textup{sym}}, \R^{d \times d}_{\textup{sym}})$ 
is a linear and coercive 4th order elasticity tensor with Lam\'e-constants $\mu = \frac{E}{2(1+\nu)}$ and $\lambda = \frac{E \nu}{(1+\nu)(1-2\nu)}$ with elasticity module $E=2900$ and shear module $\nu = 0.4$.

In each iteration of the outer approximation algorithm \ref{alg:Reps}, the two optimization problems \eqref{eq:Pk} and \eqref{eq:Qk} have to be solved. 
The first one is a linear-quadratic optimal control problem with finitely many integral inequality constraints, whose necessary and sufficient optimality conditions read as follows:
$\bar u \in \UU$ is optimal, iff there exist $\bar y, \bar p \in H^1_0(\Omega)$, and $\mu_1,\dots,\mu_k \in \R$ such that
\begin{subequations}\label{eq:nock}
    \begin{gather}
        - \laplace y_k = u_k + f \;\; \text{in } H^{-1}(\Omega), \quad 		
        - \laplace p_k =  y_k - y_d \;\; \text{in } H^{-1}(\Omega), \\
        p_k + \alpha(u_k - u_d) + \sum_{i=1}^k \mu_i \,\div \varphi_i = 0, \label{eq:gradeqk}\\
        \mu_i \geq 0, \quad \mu_i \Big( \int_\Omega u_k \,\div \varphi_i \, \d x - \varepsilon\, a[\varphi_i, \varphi_i]  \Big)= 0,
        \quad \int_\Omega u_k \,\div \varphi_i \, \d x - \varepsilon\, a[\varphi_i, \varphi_i]  \leq 1 \quad i = 1, \ldots, k.
    \end{gather}
\end{subequations}
Herein, $\mu_1, \ldots, \mu_k \in \R$ denote the Lagrange multipliers for the inequality constraints. Note that $\varphi_1, \ldots, \varphi_k$ are given data within the 
computation of $u_k$. 
In order to compute the next  cutting plane, \eqref{eq:Qk} has to be solved, whose necessary and sufficient conditions are 
given by the following variational inequality (VI)
\begin{equation}\label{eq:obst}
    \varphi_{k+1} \in \KK,\quad 
    \varepsilon\, a[\varphi_{k+1}, v - \varphi_{k+1}] = \int_\Omega u_k\, \div (v - \varphi_{k+1}) \,\d x \quad \forall\, v\in \KK,
\end{equation}
where $\KK \coloneqq \{ v\in H^1_0(\Omega;\R^d) \colon \|v\|_\infty \leq 1\}$. We observe that \eqref{eq:obst} is a VI of the first kind, which is similar 
to the classical obstacle problem.  Note that, this time, $u_k$ is given data.

In order to solve \eqref{eq:nock} and \eqref{eq:obst} numerically, 
we discretize the equations by means of conforming finite elements on a fixed triangular mesh. To be more precise, we use first-order Lagrangian finite elements, i.e., 
continuous and piecewise linear ansatz functions, for $y_k$, $p_k$, $\varphi_k$, and the respective test functions, while $u_k$ is discretized with piecewise 
constant ansatz functions. Observe that this choice is in accordance with \eqref{eq:gradeqk}, since $\div \varphi_i$ is piecewise constant, if 
$\varphi_i$ is discretized by means of first-order Lagrangian finite elements. The additional inhomogeneity $f$ as well as the desired control $u_d$ are also discretized 
with piecewise constants, while we employ the first-order Lagrange interpolation for the desired state $y_d$.
For the mesh, we use a triangulation of Friedrichs-Keller type with mesh size $h = \frac{\sqrt{2}}{50}$. 

Using a nonlinear complimentarity function, \eqref{eq:nock} and \eqref{eq:obst} can be reformulated as nonsmooth equations, that can be solved by 
the semismooth Newton method. In case of \eqref{eq:nock}, this even works in function space, cf., e.g., \cite[Section~5.2]{Ulb11} and \cite{IK08}. 
Concerning \eqref{eq:obst}, an additional smoothing step or regularization 
is necessary to obtain the norm gap that is needed for Newton differentiability, see for instance \cite[Section~9.2]{Ulb11} for details. For the solution of the discretized finite dimensional 
counterpart of \eqref{eq:obst}, this smoothing step is however not necessary, so we do without this additional step, as the mesh independence of the method 
is not in the focus of our work. We implemented the semismooth Newton method to solve the discretized counterparts to \eqref{eq:nock} and \eqref{eq:obst} 
in form of an active set method, see \cite{HIK02}.

As expected, driving $\eps \searrow 0$ is a critical issue, 
because large $\eps$ lead to a large difference between $\TV_\eps(u)$ and $\TV(u)$ and small $\eps$ lead to a small coercivity constant 
of the bilinear form in the left hand side of the VI in \eqref{eq:obst}, which results in a slow convergence of the active set method for \eqref{eq:obst}. 
To resolve this issue, we apply a path-following approach for the reduction of the regularization parameter $\varepsilon$. 
In our numerical tests, it turned out to be advantageous to reduce $\varepsilon$ in the outer loop of Algorithm~\ref{alg:Reps} and not during the active set method for \eqref{eq:obst}
in step \ref{algReps3} of Algorithm~\ref{alg:Reps}. To be more precise, we reduce $\eps$ in each iteration of \cref{alg:Reps} until the desired value of $\eps$ is reached 
and use the former solution of \eqref{eq:obst} as initial guess for the active set method for \eqref{eq:obst} in the next iterate. 
Moreover, we update the right-hand sides of the former computed cutting planes, i.e., $F_\varepsilon(u, \varphi_i) \d x \leq 1$
for all $i \in \N$ with $i \leq k$, in each iteration by inserting the current value of $\eps$ to obtain tighter cutting planes.
This procedure reduces the number of iterations that are needed by the semismooth Newton method to solve \eqref{eq:obst} significantly.

We applied algorithm \cref{alg:Reps} to two instances of \eqref{eq:P}. The results are presented in the next two sections.

\subsection{Numerical example with known exact optimal solution} \label{sec:numeric_exact_solution}
For our first numerical example, we use the exact optimal solution to \eqref{eq:P} that is determined in \cref{sec:exact_solution}. 
We start \Cref{alg:Reps} with the choice $\eps = 10^{-5}$ and reduce $\eps$ in each iteration of \cref{alg:Reps} by the factor $0.5$ 
until we reach the desired value of $\eps = \num{7.8e-08}$. Only after that, we use the termination criterion from \cref{alg:Reps} with a tolerance of $\texttt{tol} = 10^{-2}$, 
that is, we let the algorithm terminate after the current iterate $u_k$ fulfills $\TV_\eps(u_k) = F_\eps(u_k,\phi_{k+1}) \leq 1 + \texttt{tol}$. 
We initialize the active set method for the solution of \eqref{eq:Pk} and \eqref{eq:Qk} in the first iteration of \cref{alg:Reps} with $u_{\text{init}} \equiv 0$ and $\phi_{\text{init}} \equiv 0$, 
respectively, and with the optimal solutions $u_{k-1}$ and $\phi_k$ from the former iteration of \cref{alg:Reps} in the subsequent iterations. 
\Cref{alg:Reps} has performed $8$ iterations until the termination criterion was fulfilled. The results are stated in \cref{tab:exact_solution}, where \texttt{it}\eqref{eq:Pk} and 
\texttt{it}\eqref{eq:Qk} denote the number of iterations the active set method has performed to solve \eqref{eq:Pk} and \eqref{eq:Qk} in the current iteration of \cref{alg:Reps}.
Moreover, \texttt{err} denotes the relative error in the control and \texttt{eoc} the experimental order of convergence w.r.t.\ $\varepsilon$, i.e., 
\begin{equation*}
    \texttt{err}_k \coloneqq \frac{\| u_k - \bar{u}\|_{L^2(\Omega)}}{\| \bar{u} \|_{L^2(\Omega)}}, \quad 
    \texttt{eoc}_k \coloneqq \frac{\log(\texttt{err}_{k-1}) - \log(\texttt{err}_k)}{\log(\varepsilon_{k-1}) - \log(\varepsilon_k)}.
\end{equation*}
Moreover, we have computed a lower bound for $\TV(u_k)$ by
\begin{align*}
	 \TV(u_k) \geq \int_{\Omega} u_k \div \phi_{k+1} \dd x = \TV_\eps(u_k) + \frac{\eps}{2} a[\phi_{k+1},\phi_{k+1}].
\end{align*}

\begin{table}[h!]
	\centering
	\caption{Results obtained by applying \cref{alg:Reps} to the instance from \cref{sec:numeric_exact_solution} with an additional reduction of $\eps$ in each iteration.}	
    \vspace*{1ex}
	\resizebox{\textwidth}{!}{
		\begin{tabular}{c|cccccccc}
			\toprule
			$k$ & $\eps $ & $J(u_k)$ & \texttt{it}\eqref{eq:Pk} & \texttt{it}\eqref{eq:Qk} & $\TV_\eps(u_k)$ & $\TV(u_k) \geq$ & 
			\texttt{err} & \texttt{eoc} \\
			\midrule
			0 & \num{1.0e-05} & 7.753 & 1 & 6 & 0.69112 & 1.22532 & \num{2.7856e-01} &  -- \\
			1 & \num{5.0e-06} & 7.753 & 1 & 4 & 0.97830 & 1.30346 & \num{2.7856e-01} &  0.0 \\ 
			2 & \num{2.5e-06} & 7.753 & 3 & 3 & 1.01212 & 1.22899 & \num{1.9501e-01} &  0.5144 \\
			3 & \num{1.3e-06} & 7.754 & 5 & 2 & 1.00357 & 1.13971 & \num{1.2431e-01} &  0.6886 \\
			4 & \num{6.3e-07} & 7.755 & 5 & 3 & 1.00237 & 1.08947 & \num{8.5259e-02} &  0.5205 \\ 
			5 & \num{3.1e-07} & 7.755 & 4 & 4 & 1.00079 & 1.05069 & \num{5.9808e-02} &  0.5000 \\
			6 & \num{1.6e-07} & 7.756 & 5 & 3 & 1.00147 & 1.03463 & \num{5.1465e-02} &  0.2272 \\
			7 & \num{7.8e-08} & 7.756 & 4 & 4 & 1.00574 & 1.03226 & \num{4.7437e-02} & 0.1134 \\
			\bottomrule
		\end{tabular}
	}\label{tab:exact_solution}
\end{table}

\begin{figure}[h!]
	\centering
	\begin{minipage}{\textwidth}
		\centering
		\begin{minipage}[][][b]{0.49\textwidth}
			\centering
			\subfigure[lof][Exact optimal solution $\bar{u}$]{\includegraphics[height=4cm]{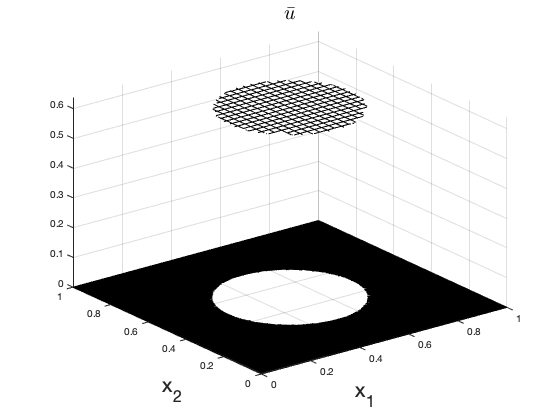}}
		\end{minipage}
		\begin{minipage}[][][b]{0.49\textwidth}
			\centering
			\subfigure[lof][Optimal solution $u_k$]{\includegraphics[height=4cm]{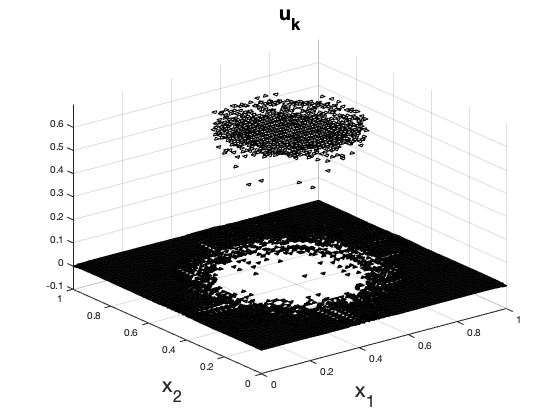}}
		\end{minipage}
	\end{minipage}
	\caption{Plots of the exact solution $\bar{u}$ to \eqref{eq:P} corresponding to the instance from \cref{sec:numeric_exact_solution} and the associated solution $u_k$ to \eqref{eq:Peps} with $\eps = \num{7.8e-08}$ obtained by applying \cref{alg:Reps}.}
	\label{fig:exact_solution}
\end{figure}

We observe that, until iteration $k=5$, the experimental order of convergence is in good agreement with the theoretical predictions of Theorem~\ref{thm:convrate}
and Corollary~\ref{cor:convrate}, respectively.
Note that the assumptions of Corollary~\ref{cor:convrate} are fulfilled, since the exact solution from Section~\ref{sec:exact_solution} satisfies 
$\bar\Phi \in H^1_0(\Omega;\R^d) = \HH$.
An explanation for the decrease of the convergence rate in the last two iterations is the additional error induced by the discretization, which only becomes visible for small values of $\varepsilon$.

\subsection{Numerical example without known exact optimal solution} \label{sec:numeric_without_exact_solution}

As a second numerical example, we consider $f\equiv 0$ and the functions $u_d, y_d \in C^\infty_0(\Omega)$ defined by
\begin{equation*}
	u_d(x) \coloneqq 2\, c\, \pi^2\, \sin(\pi x_1) \cos(\pi x_2)
	\quad \text{and} \quad 
	y_d(x) \coloneqq c\, \sin(\pi x_1) \cos(\pi x_2),
\end{equation*}
where we choose $c \coloneqq \frac{2}{\TV(\tilde{u})}$ with $\tilde{u}(x) \coloneqq 2 \pi^2\, \sin(\pi x_1) \cos(\pi x_2)$. 
This yields $- \Delta y_d = u_d$ and $\TV(u_d) = 2$ and, in particular, 
$\bar{u} = u_d$ would be an optimal solution to \eqref{eq:P}, if we replace the constraint $\TV(u) \leq 1$ by $\TV(u) \leq 2$.
We again start \cref{alg:Reps} with $\eps = 10^{-5}$ and reduce $\eps$ in each iteration of \cref{alg:Reps} by the factor $0.5$ until we reach the desired value of $\eps = \num{1.6e-07}$. 
This value is differing from the desired value in \cref{sec:numeric_exact_solution}, because the active set method applied to \eqref{eq:Qk} does not converge for $\eps = \num{7.8e-08}$
within the maximum number of iterations. After reaching the desired value for $\eps$, we use the same termination criterion with a tolerance of $\texttt{tol} = 10^{-2}$ 
as in \cref{sec:numeric_exact_solution}, that is, we again terminate if $\TV_\eps(u_k) = F_\eps(u_k,\phi_{k+1}) \leq 1 + 10^{-2}$. 
We also use the same initalizations for the active set method for \eqref{eq:Pk} and \eqref{eq:Qk} as in \cref{sec:numeric_exact_solution}. 
\Cref{alg:Reps} performes $7$ iterations until the termination criterion is fulfilled. The results are presented in \cref{tab:without_exact_solution}, where we list the same information as in \cref{tab:exact_solution} except for the relative errors, which are not known, because we do not know the corresponding optimal solution to \eqref{eq:P}. 
For the same reason, we provide the plot of the desired control $u_d$ in \cref{fig:without_exact_solution} instead of the optimal solution $\bar{u}$ to \eqref{eq:P} 
for the comparison to the plot of the solution $u_k$ to \eqref{eq:Peps} with $\eps = \num{1.6e-07}$ obtained by \cref{alg:Reps}.

\begin{table}[h!]
	\centering
	\caption{Results obtained by applying \cref{alg:Reps} to the instance from \cref{sec:numeric_without_exact_solution} with an additional reduction of $\eps$ in each iteration.}	
	\vspace*{1ex}
		\begin{tabular}{c|cccccc}
			\toprule
			$k$ & $\eps $ & $ J(u_k)$ & \texttt{it}\eqref{eq:Pk} & \texttt{it}\eqref{eq:Qk} & $\TV_\eps(u_k)$ & $\TV(u_k) \geq$ \\
			\midrule
			0 & \num{1.0e-05} & 0.113 & 1 & 1 & 0.64278 & 1.28556 \\
			1 & \num{5.0e-06} & 0.113 & 1 & 12 & 1.07743 & 1.61449 \\ 
			2 & \num{2.5e-06} & 0.115 & 3 & 6 & 1.00820 & 1.31846  \\
			3 & \num{1.3e-06} & 0.117 & 3 & 6 & 1.00330 & 1.17588  \\
			4 & \num{6.3e-07} & 0.118 & 5 & 6 & 1.00158 & 1.09302  \\
			5 & \num{3.1e-07} & 0.119 & 4 & 6 & 1.00423 & 1.05908 \\
			6 & \num{1.6e-07} & 0.119 & 4 & 7 & 1.00319 & 1.03100  \\
			\bottomrule
		\end{tabular}
	\label{tab:without_exact_solution}
\end{table}

\begin{figure}[h!]
	\centering
	\begin{minipage}{\textwidth}
		\centering
		\begin{minipage}[][][b]{0.49\textwidth}
			\centering
			\subfigure[lof][Desired control $u_d$]{\includegraphics[height=4cm]{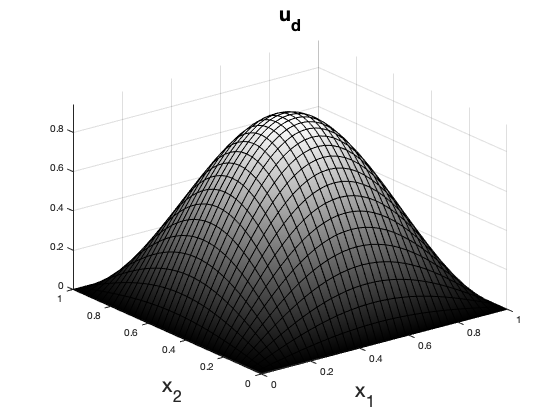}}
		\end{minipage}
		\begin{minipage}[][][b]{0.49\textwidth}
			\centering
			\subfigure[lof][Optimal solution $u_k$]{\includegraphics[height=4cm]{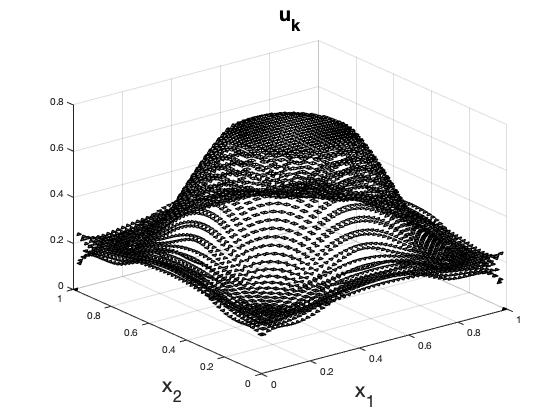}}
		\end{minipage}
	\end{minipage}
	\caption{Plots of the desired control $u_d$ from \eqref{eq:P} corresponding to the instance from \cref{sec:numeric_without_exact_solution} and the associated solution $u_k$ to \eqref{eq:Peps} with $\eps = \num{1.6e-07}$ obtained by applying \cref{alg:Reps}.}
	\label{fig:without_exact_solution}
\end{figure}

\section*{Acknowledgement}

We are very grateful to Christian Kreuzer (TU Dortmund) for pointing out reference \cite{CMcI10} to us.

\renewcommand*{\bibfont}{\footnotesize}
\printbibliography

\end{document}